\documentclass[12pt,a4paper]{amsart}

\usepackage[abbrev]{amsrefs}
\usepackage{amscd}

\theoremstyle{plain}
  \newtheorem{thm}{Theorem}[section]
  \newtheorem{lem}[thm]{Lemma}
  
  \newtheorem{cor}[thm]{Corollary}
  \newtheorem{prop}[thm]{Proposition}
  \newtheorem{clm}[thm]{Claim}

\theoremstyle{definition}
  \newtheorem{defn}[thm]{Definition}

\theoremstyle{remark}

  \newtheorem*{ack}{Acknowledgment}

\numberwithin{equation}{section}

\DeclareMathOperator{\diam}{diam}
\DeclareMathOperator{\supp}{supp}

\DeclareMathOperator{\ObsDiam}{ObsDiam}
\DeclareMathOperator{\me}{me}
\DeclareMathOperator{\Sep}{Sep}
\DeclareMathOperator{\id}{id}
\DeclareMathOperator{\dconc}{{\it d}_{{\rm conc}}}
\newcommand{\Lip}{\mathcal{L}{\it ip}}

\newcommand{\Lo}{\mathcal{L}_1}
\newcommand{\CP}{\C P}

\newcommand{\cL}{\mathcal{L}}
\newcommand{\cM}{\mathcal{M}}
\newcommand{\cP}{\mathcal{P}}

\newcommand{\cX}{\mathcal{X}}

\newcommand{\field}[1]{\mathbb{#1}}

\newcommand{\C}{\field{C}}
\newcommand{\R}{\field{R}}

\begin{document}

\title[Metric measure Limits]
{Metric measure limits of spheres and complex projective spaces}

\thanks{The author is partially supported by a Grant-in-Aid
for Scientific Research from the Japan Society for the Promotion of Science}

\begin{abstract}
  We study the limits of sequences of spheres and
  complex projective spaces
  with unbounded dimensions.
  A sequence of spheres (resp.~complex projective spaces)
  either is a L\'evy family, infinitely dissipates,
  or converges to (resp.~the Hopf quotient of)
  a virtual infinite-dimensional Gaussian space,
  depending on the size of the spaces.
  These are the first discovered examples with the property
  that the limits are drastically different from
  the spaces in the sequence.
  For the proof, we introduce a metric on Gromov's compactification
  of the space of metric measure spaces.
\end{abstract}

\author{Takashi Shioya}

\address{Mathematical Institute, Tohoku University, Sendai 980-8578,
  JAPAN}

\date{\today}


\maketitle

\section{Introduction}
\label{sec:intro}

Gromov \cite{Gromov}*{\S 3$\frac{1}{2}$} developed the metric measure geometry
based on the idea of concentration of measure phenomenon
due to L\'evy and Milman (see \cite{Levy,Mil:heritage,Mil:inf-dim,Mil:hom-sp}).
This is particularly useful to study a family of spaces
with unbounded dimensions.
A \emph{metric measure space} (or \emph{mm-space} for short)
is a triple $(X,d_X,\mu_X)$, where $(X,d_X)$ is a complete separable
metric space and $\mu_X$ a Borel probability measure on $X$.
Gromov defined the \emph{observable distance}, say $\dconc(X,Y)$,
between two mm-spaces $X$ and $Y$
by the difference between $1$-Lipschitz functions on $X$ and
those on $Y$, and studied the geometry of the space of mm-spaces, say $\cX$,
with metric $\dconc$.
(In \cite{Gromov}, the observable distance function is denoted by
$\underline{H}_1\cL\iota_1$.)
The observable distance is much more useful than
the Gromov-Hausdorff distance to study a sequence of spaces
with unbounded dimensions.
We say that a sequence of mm-spaces $X_n$, $n=1,2,\dots$, \emph{concentrates}
to an mm-space $X$ if $X_n$ $\dconc$-converges to $X$ as $n\to\infty$.
We have a specific natural compactification, say $\Pi$,
of $(\cX,\dconc)$.  The space $\Pi$ consists of pyramids,
where a \emph{pyramid} is a directed subfamily of $\cX$
with respect to a natural order relation $\prec$,
called the \emph{Lipschitz order}.
$X \prec Y$ holds if there exists a $1$-Lipschitz
continuous map from $Y$ to $X$ that pushes $\mu_Y$ forward to $\mu_X$.
For a given mm-space $X \in \cX$, the set of $X' \in \cX$ with $X' \prec X$
is a pyramid, denoted by $\cP_X$.
We call $\cP_X$ the \emph{pyramid associated with $X$}.
The space $\Pi$ has a natural compact topology such that
the map
\[
\iota : \cX \ni X \longmapsto \cP_X \in \Pi
\]
is a topological embedding map.
The image $\iota(\cX)$ is dense in $\Pi$
and so $\Pi$ is a compactification of $\cX$.
It follows that $X \prec Y$ if and only if $\cP_X \subset \cP_Y$,
namely the Lipschitz order $\prec$ on $\cX$
extends to the inclusion relation on $\Pi$.
$\cX$ itself is a maximal element of $\Pi$,
and a one-point mm-space corresponds to a minimal element of $\Pi$.
A sequence of mm-spaces is a \emph{L\'evy family}
if and only if it concentrates to a one-point mm-space.
A sequence of mm-spaces \emph{infinitely dissipates}
if and only if the sequence of the associated pyramids
converges to the maximal element $\cX$.

It is interesting to study concrete examples of nontrivial
sequences of mm-spaces and their limits (in $\Pi$ and in $\cX$),
where `nontrivial' means that it neither is a L\'evy family,
dissipates, nor $\square$-converges,
where $\square$ denotes the box distance function on $\cX$,
which is an elementary metric on $\cX$ and satisfies $\dconc \le \square$.
We remark that there are very few nontrivial examples
that are studied in detail before.
All such known examples are of the type of product spaces
(see \cite{Gromov}*{\S 3$\frac12$.49,56}).
In this paper, we study two examples of the non-product type,
spheres and complex projective spaces with unbounded dimensions.
Those are also the first discovered examples of sequence
with the property that the limit space is drastically different from
the spaces in the sequence.

We present some definitions needed to state our main theorems.
The precise definitions are given in \S\ref{sec:Gaussian-Hopf}.
Let $\gamma^\infty$ be the infinite-dimensional standard Gaussian measure
on $\R^\infty$.
We call $\Gamma^\infty := (\R^\infty,\|\cdot\|_2,\gamma^\infty)$
the \emph{infinite-dimensional standard Gaussian space},
where $\|\cdot\|_2$ denotes the $l_2$ norm on $\R^\infty$ (which takes
values in $[\,0,+\infty\,]$).
Note that $\gamma^\infty$ is not a Borel measure
with respect to the $l_2$ norm (cf.~\cite{Bog:Gaussian}*{\S 2.3})
and that $\Gamma^\infty$ is not an mm-space.
Nevertheless we have the associated pyramid $\cP_{\Gamma^\infty}$.
We call $\cP_{\Gamma^\infty}$
the \emph{virtual infinite-dimensional standard Gaussian space}.
In the same way, we consider the infinite-dimensional centered Gaussian measure 
$\gamma^\infty_{\lambda^2}$ with variance $\lambda^2$,
$\lambda > 0$, and define the virtual infinite-dimensional Gaussian space
$\cP_{\Gamma^\infty_{\lambda^2}}$ as a pyramid.
$\cP_{\Gamma^\infty_{\lambda^2}}$ coincides with the scale change of
$\cP_{\Gamma^\infty}$ of factor $\lambda$.
We consider the Hopf action on $\Gamma^\infty_{\lambda^2}$ by
identifying $\R^\infty$ with $\C^\infty$.
The Hopf action is isometric with respect to the $l_2$ norm
and also preserves $\gamma^\infty_{\lambda^2}$.
The quotient space $\Gamma^\infty_{\lambda^2}/S^1$
has a natural measure and a metric.
We also have the associated pyramid $\cP_{\Gamma^\infty_{\lambda^2}/S^1}$.
Let $S^n(r)$ be the $n$-dimensional sphere in $\R^{n+1}$ centered at the origin
and of radius $r > 0$.  We equip the Riemannian distance function
or the restriction of the Euclidean distance function
with $S^n(r)$.
We also equip the normalized volume measure with $S^n(r)$.
Then $S^n(r)$ is an mm-space.
We consider the Hopf quotient
\[
\CP^n(r) := S^{2n+1}(r)/S^1
\]
that has a natural mm-structure induced from that of $S^{2n+1}(r)$
(see \S\ref{ssec:quotient}).
This is topologically an $n$-dimensional complex projective space.
Note that, if the distance function on $S^{2n+1}(r)$ is assumed to be Riemannian,
then the distance function on $\CP^n(r)$ coincides with that induced from
the Fubini-Study metric scaled with factor $r$.

One of our main theorems in this paper is stated as follows.

\begin{thm} \label{thm:main}
  Let $\{r_n\}_{n=1}^\infty$ be a given sequence of positive real numbers,
  and let $\lambda_n := r_n/\sqrt{n}$
  {\rm(}resp.~$\lambda_n := r_n/\sqrt{2n+1}${\rm)}.
  Then we have the following {\rm(1)}, {\rm(2)}, and {\rm(3)}.
  \begin{enumerate}
  \item $\lambda_n \to 0$ as $n\to\infty$ if and only if
    $\{S^n(r_n)\}_{n=1}^\infty$
    {\rm(}resp.~$\{\CP^n(r_n)\}_{n=1}^\infty${\rm)}
    is a L\'evy family.
  \item $\lambda_n \to +\infty$ as $n\to\infty$ if and only if
    $\{S^n(r_n)\}_{n=1}^\infty$
    {\rm(}resp.\\ $\{\CP^n(r_n)\}_{n=1}^\infty${\rm)}
    infinitely dissipates.
  \item Assume that $\lambda_n \to \lambda$ as $n\to\infty$
    for a positive real number $\lambda$.
    Then, as $n\to\infty$,
    $\cP_{S^n(r_n)}$ {\rm(}resp.~$\cP_{\CP^n(r_n)}${\rm)} converges to
    $\cP_{\Gamma^\infty_{\lambda^2}}$
    {\rm(}resp.~$\cP_{\Gamma^\infty_{\lambda^2}/S^1}${\rm)}.
  \end{enumerate}
  {\rm(1)} and {\rm(2)} both hold for the Riemannian metric
  {\rm(}resp.~the scaled Fubini-Study metric{\rm)}
  and also for the Euclidean distance function
  {\rm(}resp.~the distance induced from the Euclidean{\rm)}.
  {\rm(3)} holds only for for the Euclidean distance function
  {\rm(}resp.~the distance induced from the Euclidean{\rm)}.
\end{thm}

Theorem \ref{thm:main} is analogous to phase transition phenomena
in statistical mechanics.

If $r_n$ is bounded away from zero,
then $\{S^n(r_n)\}_{n=1}^\infty$ and $\{\CP^n(r_n)\}_{n=1}^\infty$
both have no $\square$-convergent subsequence
(see Proposition \ref{prop:no-box}).
The theorem also holds for any subsequence of $\{n\}$.
We have the same statement as in Theorem \ref{thm:main}
also for real and quotanionic projective spaces
in the same way.

(1) of Theorem \ref{thm:main} follows essentially from
the works of L\'evy and Milman.
For (1), we give some fine estimates of the observable diameter
by using the normal law \`a la L\'evy (Theorem \ref{thm:normal}).
(2) follows from a discussion using the Maxwell-Boltzmann distribution law
(Proposition \ref{prop:MB-law}).
(3) is the most important part of the theorem.
It follows from the Maxwell-Boltzmann distribution law that
the limit pyramid of (the Hopf quotient of) spheres
contains the (Hopf quotient of) virtual Gaussian space.
For the proof of the reverse inclusion,
we define a metric $\rho$
on the space $\Pi$ of pyramids compatible with the topology on $\Pi$
(see Definition \ref{defn:metric-Pi})
that satisfies the following theorem.

\begin{thm} \label{thm:rho-dconc}
  For any two mm-spaces $X$ and $Y$, we have
  \[
  \rho(\cP_X,\cP_Y) \le \dconc(X,Y),
  \]
  i.e., the embedding map
  $\iota :\cX \ni X \mapsto \cP_X \in \Pi$ is $1$-Lipschitz continuous
  with respect to $\dconc$ and $\rho$.
\end{thm}

Applying this theorem, we prove the reverse inclusion.

Note that Gromov \cite{Gromov}*{\S 3$\frac12$} gave only the notion of
convergence of pyramids and did not define the topology on $\Pi$.
Note also that there exists no metric on $\Pi$ (strongly)
equivalent to $\dconc$.
In fact we have the following as a consequence of Theorem \ref{thm:main}.

\begin{prop} \label{prop:dconc}
  There exist mm-spaces $X_n$ and $Y_n$, $n=1,2,\dots$, such that
  \begin{enumerate}
  \item $\dconc(X_n,Y_n)$ is bounded away from zero;
  \item the associated pyramids $\cP_{X_n}$ and $\cP_{Y_n}$ both converge
    to a common pyramid as $n\to\infty$.
  \end{enumerate}
\end{prop}

\begin{ack}
  The author would like to thank Professors Takefumi Kondo and Asuka Takatsu
  for valuable discussions.
  He also thanks to Professors Tomohiro Fukaya and Ayato Mitsuishi
  for useful comments.
\end{ack}

\section{Preliminaries}

In this section, we give the definitions and the facts
stated in \cite{Gromov}*{\S 3$\frac12$}.
In \cite{Gromov}*{\S 3$\frac12$}, many details are omitted.
We refer to \cite{Funano:thesis,
FS, Ollivier:SnCPn, Shioya:book} for the details.
The reader is expected to be familiar with basic measure theory
and metric geometry (cf.~\cite{Kechris, Bil, Bog, BBI}).

\subsection{mm-Isomorphism and Lipschitz order}

\begin{defn}[mm-Space]
  Let $(X,d_X)$ be a complete separable metric space
  and $\mu_X$ a Borel probability measure on $X$.
  We call the triple $(X,d_X,\mu_X)$ an \emph{mm-space}.
  We sometimes say that $X$ is an mm-space, in which case
  the metric and the measure of $X$ are respectively indicated by
  $d_X$ and $\mu_X$.
\end{defn}

\begin{defn}[mm-Isomorphism]
  Two mm-spaces $X$ and $Y$ are said to be \emph{mm-isomorphic}
  to each other if there exists an isometry $f : \supp\mu_X \to \supp\mu_Y$
  such that $f_*\mu_X = \mu_Y$,
  where $f_*\mu_X$ is the push-forward of $\mu_X$ by $f$.
  Such an isometry $f$ is called an \emph{mm-isomorphism}.
  Denote by $\cX$ the set of mm-isomorphism classes of mm-spaces.
\end{defn}

Any mm-isomorphism between mm-spaces is automatically surjective,
even if we do not assume it.
Note that $X$ is mm-isomorphic to $(\supp\mu_X,d_X,\mu_X)$.

\emph{We assume that an mm-space $X$ satisfies
\[
X = \supp\mu_X
\]
unless otherwise stated.}

\begin{defn}[Lipschitz order] \label{defn:dom}
  Let $X$ and $Y$ be two mm-spaces.
  We say that $X$ (\emph{Lipschitz}) \emph{dominates} $Y$
  and write $Y \prec X$ if
  there exists a $1$-Lipschitz map $f : X \to Y$ satisfying
  \[
  f_*\mu_X = \mu_Y.
  \]
  We call the relation $\prec$ on $\cX$ the \emph{Lipschitz order}.
\end{defn}

The Lipschitz order $\prec$ is a partial order relation on $\cX$

\subsection{Observable diameter}

The observable diameter is one of the most fundamental invariants
of an mm-space.

\begin{defn}[Partial and observable diameter]
  Let $X$ be an mm-space.
  For a real number $\alpha$, we define
  the \emph{partial diameter
    $\diam(X;\alpha) = \diam(\mu_X;\alpha)$ of $X$}
  to be the infimum of $\diam A$,
  where $A \subset X$ runs over all Borel subsets
  with $\mu_X(A) \ge \alpha$ and $\diam A$ denotes the diameter of $A$.
  For a real number $\kappa > 0$, we define
  the \emph{observable diameter of $X$} to be
  \begin{align*}
    \ObsDiam(X;-\kappa) &:= \sup\{\;\diam(f_*\mu_X;1-\kappa) \mid\\
    &\qquad\qquad\text{$f : X \to \R$ is $1$-Lipschitz continuous}\;\}.
  \end{align*}
\end{defn}


\begin{defn}[L\'evy family]
  A sequence of mm-spaces $X_n$, $n=1,2,\dots$,
  is called a \emph{L\'evy family} if
  \[
  \lim_{n\to\infty} \ObsDiam(X_n;-\kappa) = 0
  \]
  for any $\kappa > 0$.
\end{defn}

\begin{prop} \label{prop:ObsDiam-dom}
  If $X \prec Y$, then
  \[
  \ObsDiam(X;-\kappa) \le \ObsDiam(Y;-\kappa)
  \]
  for any $\kappa > 0$.
\end{prop}

\subsection{Separation distance}

\begin{defn}[Separation distance]
  Let $X$ be an mm-space.
  For any real numbers $\kappa_0,\kappa_1,\cdots,\kappa_N > 0$
  with $N\geq 1$,
  we define the \emph{separation distance}
  \[
  \Sep(X;\kappa_0,\kappa_1, \cdots, \kappa_N)
  \]
  of $X$ as the supremum of $\min_{i\neq j} d_X(A_i,A_j)$
  over all sequences of $N+1$ Borel subsets $A_0,A_2, \cdots, A_N \subset X$
  satisfying that $\mu_X(A_i) \geq \kappa_i$ for all $i=0,1,\cdots,N$,
  where $d_X(A_i,A_j) := \inf_{x\in A_i,y\in A_j} d_X(x,y)$.
  If there exists no sequence $A_0,\dots,A_N \subset X$
  with $\mu_X(A_i) \ge \kappa_i$, $i=0,1,\cdots,N$, then
  we define
  \[
  \Sep(X;\kappa_0,\kappa_1, \cdots, \kappa_N) := 0.
  \]
\end{defn}


\begin{lem} \label{lem:Sep-prec}
  Let $X$ and $Y$ be two mm-spaces.
  If $X$ is dominated by $Y$, then we have,
  for any real numbers $\kappa_0,\dots,\kappa_N > 0$,
  \[
  \Sep(X;\kappa_0,\dots,\kappa_N) \le \Sep(Y;\kappa_0,\dots,\kappa_N).
  \]
\end{lem}


\begin{prop} \label{prop:ObsDiam-Sep}
  For any mm-space $X$ and any real numbers $\kappa$ and $\kappa'$
  with $\kappa > \kappa' > 0$, we have
  \begin{align}
    \tag{1} &\ObsDiam(X;-2\kappa) \le \Sep(X;\kappa,\kappa),\\
    \tag{2} &\Sep(X;\kappa,\kappa) \le \ObsDiam(X;-\kappa').
  \end{align}
\end{prop}

\subsection{Box distance and observable distance}

\begin{defn}[Prokhorov distance]
  The \emph{Prokhorov distance $d_P(\mu,\nu)$ between two Borel probability
    measures $\mu$ and $\nu$ on a metric space $X$}
  is defined to be the infimum of $\varepsilon > 0$ satisfying
  \begin{equation} \label{eq:Proh}
    \mu(U_\varepsilon(A)) \ge \nu(A) - \varepsilon
  \end{equation}
  for any Borel subset $A \subset X$, where
  \[
  U_\varepsilon(A) := \{\; x \in X \mid d_X(x,A) < \varepsilon\;\}.
  \]
  \index{UepsilonA@$U_\varepsilon(A)$}
\end{defn}

The Prokhorov metric is a metrization of weak convergence of
Borel probability measures on $X$ provided that $X$ is a separable
metric space.

\begin{defn}[$\me$]
  Let $(X,\mu)$ be a measure space and $Y$ a metric space.
  For two $\mu$-measurable maps $f,g : X \to Y$, we define $\me_\mu(f,g)$
  to be the infimum of $\varepsilon \ge 0$ satisfying
  \begin{align} \label{eq:me}
    \mu(\{\;x \in X \mid d_Y(f(x),g(x)) > \varepsilon\;\}) \le \varepsilon.
  \end{align}
  We sometimes write $\me(f,g)$ by omitting $\mu$.
\end{defn}

$\me_\mu$ is a metric on the set of $\mu$-measurable maps from $X$ to $Y$
by identifying two maps if they are equal $\mu$-a.e.

\begin{lem} \label{lem:dP-me}
  Let $X$ be a topological space with a Borel probability measure $\mu$
  and $Y$ a metric space.
  For any two $\mu$-measurable maps $f,g : X \to Y$, we have
  \[
  d_P(f_*\mu,g_*\mu) \le \me_\mu(f,g).
  \]
\end{lem}

\begin{defn}[Parameter]
  Let $I := [\,0,1\,)$ and let $X$ be an mm-space.
  A map $\varphi : I \to X$ is called a \emph{parameter of $X$}
  if $\varphi$ is a Borel measurable map such that
  \[
  \varphi_*\cL^1 = \mu_X,
  \]
  where $\cL^1$ denotes the one-dimensional Lebesgue measure on $I$.
\end{defn}

Any mm-space has a parameter.

\begin{defn}[Box distance]
  We define the \emph{box distance $\square(X,Y)$ between
    two mm-spaces $X$ and $Y$} to be
  the infimum of $\varepsilon \ge 0$
  satisfying that there exist parameters
  $\varphi : I \to X$, $\psi : I \to Y$, and
  a Borel subset $I_0 \subset I$ such that
  \begin{align}
    & |\,\varphi^*d_X(s,t)-\psi^*d_Y(s,t)\,| \le \varepsilon
    \quad\text{for any $s,t \in I_0$};\tag{1}\\
    & \cL^1(I_0) \ge 1-\varepsilon,\tag{2}
  \end{align}
  where
  $\varphi^*d_X(s,t) := d_X(\varphi(s),\varphi(t))$ for $s,t \in I$.
\end{defn}

The box distance function $\square$ is a complete separable metric on $\cX$.

\begin{prop} \label{prop:box-di}
  Let $X$ be a complete separable metric space.
  For any two Borel probability measures $\mu$ and $\nu$ on $X$,
  we have
  \[
  \square((X,\mu),(X,\nu)) \le 2 \, d_P(\mu,\nu).
  \]
\end{prop}

\begin{defn}[Observable distance $\dconc(X,Y)$] \label{defn:obs-dist}
  Denote by $\Lip_1(X)$ the set of $1$-Lipschitz continuous
  functions on an mm-space $X$.
  For any parameter $\varphi$ of $X$, we set
  \[
  \varphi^*\Lip_1(X)
  := \{\;f\circ\varphi \mid f \in \Lip_1(X)\;\}.
  \]
  We define the \emph{observable distance $\dconc(X,Y)$ between
    two mm-spaces $X$ and $Y$} by
  \[
  \dconc(X,Y) := \inf_{\varphi,\psi} d_H(\varphi^*\Lip_1(d_X),\psi^*\Lip_1(d_Y)),
  \]
  where $\varphi : I \to X$ and $\psi : I \to Y$ run over all parameters
  of $X$ and $Y$, respectively,
  and where $d_H$
  is the Hausdorff distance with respect to the metric $\me_{\cL^1}$.
  We say that a sequence of mm-spaces $X_n$, $n=1,2,\dots$,
  \emph{concentrates} to an mm-space $X$ if $X_n$ $\dconc$-converges to $X$
  as $n\to\infty$.
\end{defn}

\begin{prop}
  Let $\{X_n\}_{n=1}^\infty$ be a sequence of mm-spaces.
  Then, $\{X_n\}$ is a L\'evy family if and only if $X_n$ concentrates to
  a one-point mm-space as $n\to\infty$.
\end{prop}

\begin{prop} \label{prop:dconc-box}
  For any two mm-spaces $X$ and $Y$ we have
  \[
  \dconc(X,Y) \le \square(X,Y).
  \]
\end{prop}

\subsection{Quotient space} \label{ssec:quotient}

Let $X$ be a metric space and $G$ a group acting on $X$ isometrically.
We define a pseudo-metric on the quotient space $X/G$ by
\[
d_{X/G}([x],[y]) := \inf_{x'\in [x],y'\in [y]} d_X(x',y'),
\qquad [x],[y] \in X/G.
\]
We call $d_{X/G}$ the \emph{pseudo-metric on $X/G$ induced from $d_X$}.
If every orbit in $X$ of $G$ is closed, then $d_{X/G}$ is a metric.

Assume that we have a Borel measure $\mu_X$ on $X$.
Then, we call the measure $\mu_{X/G} := \pi_*\mu_X$
the \emph{measure on $X/G$ induced from $\mu_X$},
where $\pi : X \to X/G$ is the natural projection.

\subsection{Pyramid}

\begin{defn}[Pyramid] \label{defn:pyramid}
  A subset $\cP \subset \cX$ is called a \emph{pyramid}
  if it satisfies the following (1), (2), and (3).
  \begin{enumerate}
  \item If $X \in \cP$ and if $Y \prec X$, then $Y \in \cP$.
  \item For any two mm-spaces $X, X' \in \cP$,
    there exists an mm-space $Y \in \cP$ such that
    $X \prec Y$ and $X' \prec Y$.
  \item $\cP$ is nonempty and $\square$-closed.
  \end{enumerate}
  We denote the set of pyramids by $\Pi$.

  For an mm-space $X$ we define
  \[
  \cP_X := \{\;X' \in \cX \mid X' \prec X\;\}.
  \]
  We call $\cP_X$ the \emph{pyramid associated with $X$}.
\end{defn}

It is trivial that $\cX$ is a pyramid.

In Gromov's book \cite{Gromov}, the definition of a pyramid
is only by (1) and (2) of Definition \ref{defn:pyramid}.
We here put (3) as an additional condition for the Hausdorff property
of $\Pi$.

\begin{defn}[Weak convergence] \label{defn:w-conv}
  Let $\cP_n, \cP \in \Pi$, $n=1,2,\dots$.
  We say that \emph{$\cP_n$ converges weakly to $\cP$}
  as $n\to\infty$
  if the following (1) and (2) are both satisfied.
  \begin{enumerate}
  \item For any mm-space $X \in \cP$, we have
    \[
    \lim_{n\to\infty} \square(X,\cP_n) = 0.
    \]
  \item For any mm-space $X \in \cX \setminus \cP$, we have
    \[
    \liminf_{n\to\infty} \square(X,\cP_n) > 0.
    \]
  \end{enumerate}
\end{defn}

\begin{thm} \label{thm:cpt-pyramid}
  The set $\Pi$ of pyramids is sequentially compact,
  i.e., any sequence of pyramids has a subsequence that converges weakly to
  a pyramid.
\end{thm}

\begin{prop} \label{prop:conc-pyramid}
  For given mm-spaces $X$ and $X_n$, $n=1,2,\dots$,
  the following {\rm(1)} and {\rm(2)} are equivalent to each other.
  \begin{enumerate}
  \item $X_n$ concentrates to $X$ as $n\to\infty$.
  \item $\cP_{X_n}$ converges weakly to $\cP_X$ as $n\to\infty$.
  \end{enumerate}
\end{prop}

Dissipation is an opposite notion to concentration.

\begin{defn}[Infinite dissipation]
  Let $X_n$, $n=1,2,\dots$, be mm-spaces.
  We say that $\{X_n\}$ \emph{infinitely dissipates}
  if for any real numbers $\kappa_0,\kappa_1,\dots,\kappa_N > 0$
  with $\sum_{i=0}^N \kappa_i < 1$, the separation distance
  \[
  \Sep(X_n;\kappa_0,\kappa_1,\dots,\kappa_N)
  \]
  diverges to infinity as $n\to\infty$.
\end{defn}

\begin{prop}
  Let $X_n$, $n=1,2,\dots$, be mm-spaces.
  Then, $\{X_n\}$ infinitely dissipates if and only if
  $\cP_{X_n}$ converges to $\cX$ as $n\to\infty$.
\end{prop}

Let $X$ be an mm-space.
Denote by $\Lo(X)$ the quotient of $\Lip_1(X)$ by the $\R$-action:
\[
\R \times \Lip_1(X) \ni (t,f) \mapsto t+f \in \Lip_1(X).
\]
The $\R$-action on $\Lip_1(X)$ is isometric with respect to the metric $\me$.
We denote also by `$\me$' the induced metric on $\Lo(X)$ from `$\me$'.
We see that $\me([f],[g]) = \inf_{t \in \R} \me(f+t,g)
= \me(f+t_0,g)$ for some $t_0 \in \R$.

\begin{lem} \label{lem:dGH-dconc}
  For any two mm-spaces $X$ and $Y$ we have
  \[
  d_{GH}(\Lo(X),\Lo(Y)) \le \dconc(X,Y),
  \]
  where $d_{GH}$ denotes the Gromov-Hausdorff distance function.
\end{lem}

\begin{defn}
  We say that a pyramid $\cP$ is \emph{concentrated}
  if $\{\Lo(X)\}_{X \in \cP}$ is $d_{GH}$-precompact.
\end{defn}

\begin{lem} \label{lem:concentrated}
  Let $\cP$ be a pyramid.  Then the following {\rm(1)} and {\rm(2)}
  are equivalent to each other.
  \begin{enumerate}
  \item $\cP$ is concentrated.
  \item $\cP$ is the weak limit of $\{\cP_{X_n}\}$
    for some $\dconc$-Cauchy sequence $\{X_n\}$ of mm-spaces.
  \end{enumerate}
\end{lem}

\section{Metric on the space of pyramids}

The purpose of this section is to define a metric $\rho$ on $\Pi$
compatible with weak convergence such that
the embedding map
\[
\iota : \cX \ni X \longmapsto \cP_X \in \Pi
\]
is a $1$-Lipschitz continuous with respect to
$\dconc$ on $\cX$.

\begin{defn}[$\cM(N)$, $\cM(N,R)$, $\cX(N,R)$]
  \index{XNR@$\cX(N,R)$}
  Let $N$ be a natural number and $R$ a nonnegative real number.
  Denote by $\cM(N)$ the set of Borel probability measures on $\R^N$
  equipped with the Prokhorov metric $d_P$, and set
  \[
  \cM(N,R) := \{\;\mu \in \cM(N) \mid \supp\mu \subset B^N_R\;\},
  \]
  where $B^N_R := \{\;x \in \R^N \mid \|x\|_\infty \le R\;\}$
  and $\|\cdot\|_\infty$ denotes the $l_\infty$ norm on $\R^N$.
  We define
  \[
  \cX(N,R) := \{\;(B^N_R,\|\cdot\|_\infty,\mu) \mid \mu \in \cM(N,R)\;\}.
  \]
\end{defn}

Note that $\cM(N,R)$ and $\cX(N,R)$ are compact
with respect to $d_P$ and $\square$, respectively.

\begin{defn}[$N$-Measurement]
  Let $X$ be an mm-space, $N$ a natural number,
  and $R$ a nonnegative real number.
  We define
  \begin{align*}
    \cM(X;N) &:= \{\;\Phi_*\mu_X \mid \Phi : X \to (\R^N,\|\cdot\|_\infty)
    \ \text{is $1$-Lipschitz}\;\},\\
    \cM(X;N,R) &:= \{\;\mu \in \cM(X;N) \mid
    \supp\mu \subset B^N_R\;\}.
  \end{align*}
  We call $\cM(X;N)$ (resp.~$\cM(X;N,R)$)
  the \emph{$N$-measurement} (resp.~\emph{$(N,R)$-measurement}) \emph{of} $X$.
\end{defn}

The $N$-measurement $\cM(X;N)$ is a closed subset of $\cM(N)$
and the $(N,R)$-measurement $\cM(X;N,R)$ is a compact subset of $\cM(N)$.

The following lemma is claimed in \cite{Gromov}*{\S 3$\frac12$}
without proof.
Since the lemma is important for the definition of $\rho$,
we give a sketch of proof (the detailed proof is lengthy and contained
in the book \cite{Shioya:book}).

\begin{lem}[\cite{Gromov}*{\S 3$\frac12$}] \label{lem:pyramid-conv-dH}
  For given pyramids $\cP$ and $\cP_n$, $n=1,2,\dots$,
  the following {\rm(1)} and {\rm(2)} are equivalent to each other.
  \begin{enumerate}
  \item $\cP_n$ converges weakly to $\cP$ as $n\to\infty$.
  \item For any natural number $k$,
    the set $\cP_n\cap\cX(k,k)$ Hausdorff
    converges to $\cP\cap\cX(k,k)$ as $n\to\infty$, where the Hausdorff distance
    is induced from the box metric.
  \end{enumerate}
\end{lem}

\begin{proof}[Sketch of proof]
  We prove `(1) $\implies$ (2)'.
  Suppose that $\cP_n$ converges weakly to $\cP$, but
  $\cP_n\cap\cX(k,k)$ does not Hausdorff converge to $\cP\cap\cX(k,k)$
  for some $k$.
  We then find a subsequence $\{\cP_{n_i}\}$ of $\{\cP_n\}$ in such a way that
  $\liminf_{n\to\infty} d_H(\cP_n\cap\cX(k,k),\cP\cap\cX(k,k)) > 0$.
  Since $\cX(k,k)$ is $\square$-compact
  and by replacing $\{\cP_{n_i}\}$ with a subsequence,
  $\cP_{n_i} \cap \cX(k,k)$ Hausdorff converges to
  some compact subset $\cP_\infty \subset \cX(k,k)$
  different from $\cP\cap\cX(k,k)$.
  Since any mm-space $X \in \cP_\infty$ is the limit of
  some $X_i \in \cP_{n_i} \cap \cX(k,k)$, $i=1,2,\dots$,
  the set $\cP_\infty$ is contained in $\cP$, so that
  $\cP_\infty \subset \cP\cap\cX(k,k)$.
  For any mm-space $X \in \cP \cap \cX(k,k)$,
  there is a sequence of mm-spaces $X_i \in \cP_{n_i}$ $\square$-converging
  to $X$ as $i\to\infty$.
  We are able to find a sequence of mm-spaces $X_i' \in \cX(k,k)$
  with $X_i' \prec X_i$ that $\square$-converges to $X$.
  Since $X_i' \in \cP_{n_i} \cap \cX(k,k)$,
  the space $X$ belongs to $\cP_\infty$.
  Thus we have $\cP_\infty = \cP \cap \cX(k,k)$.
  This is a contradiction.

  We prove `(2) $\implies$ (1)'.
  We assume (2).
  Let $\underline{\cP}_\infty$ be the set of the limits of
  convergent sequences of mm-spaces $X_n \in \cP_n$, and
  $\overline{\cP}_\infty$ the set of the limits of
  convergent \emph{sub}sequences of mm-spaces $X_n \in \cP_n$.
  We have $\underline{\cP}_\infty \subset \overline{\cP}_\infty$ in general.
  We shall prove $\underline{\cP}_\infty = \overline{\cP}_\infty = \cP$.

  To prove $\cP \subset \underline{\cP}_\infty$,
  we take any mm-space $X \in \cP$.
  Since $\cP \cap \bigcup_{N=1}^\infty \cX(N,N)$ is $\square$-dense in $\cP$,
  there is a sequence of mm-spaces $X_i \in \cP \cap \bigcup_{N=1}^\infty \cX(N,N)$
  that $\square$-converges to $X$.
  For each $i$ we find a natural number $N_i$ with $X_i \in \cX(N_i,N_i)$.
  By (2), there is a sequence of mm-spaces
  $X_{in} \in \cP_n \cap \cX(N_i,N_i)$, $n=1,2,\dots$,
  that $\square$-converges to $X_i$ for each $i$.
  There is a sequence $i_n \to \infty$ such that
  $X_{i_nn}$ $\square$-converges to $X$, so that
  $X$ belongs to $\underline{\cP}_\infty$.
  We obtain $\cP \subset \underline{\cP}_\infty$.

  To prove $\overline{\cP}_\infty \subset \cP$,
  we take any mm-space $X \in \overline{\cP}_\infty$.
  $X$ is approximated by
  some $\underline{X}_N = (\R^N,\|\cdot\|_\infty,\underline{\mu}_N)$,
  $\underline{\mu}_N \in \cM(X;N)$.
  It is easy to see that for any $R > 0$
  there is a unique nearest point projection
  $\pi_R : \R^N \to B^N_R$ with respect to the $l_\infty$ norm.
  $\pi_R$ is $1$-Lipschitz continuous with respect to the $l_\infty$ norm.
  Since $(\pi_R)_*\underline{\mu}_N \to \underline{\mu}_N$ weakly
  as $R \to +\infty$,
  $X$ is approximated by some $X' \in \cX(N,R)$ with $X' \prec X$.
  By the $\square$-closedness of $\cP$,
  it suffices to prove that $X'$ belongs to $\cP$.
  It follows from $X \in \overline{\cP}_\infty$ that
  there are sequences $n_i \to \infty$ and $X_i \in \cP_{n_i}$, $i=1,2,\dots$,
  such that $X_i$ $\square$-converges to $X$.
  We find a sequence of mm-spaces
  $X_i'$ with $X_i' \prec X_i$ that $\square$-converges to $X'$.
  We are also able to find
  a sequence $X_i'' \in \cX(N,R)$ such that
  $X_i'' \prec X_i'$ for any $i$ and $X_i''$ converges to $X'$ as $i\to\infty$.
  Since $\cP_{n_i}$ is a pyramid, $X_i''$ belongs to $\cP_{n_i}$.
  By (2), $X'$ is an element of $\cP$.
  We thus obtain $\cP = \underline{\cP}_\infty = \overline{\cP}_\infty$.

  We prove the weak convergence $\cP_n \to \cP$.
  Let us verify the first condition of Definition \ref{defn:w-conv}.
  Take any mm-space $X \in \cP$.
  Since $X \in \underline{\cP}_\infty$,
  there is a sequence of mm-spaces $X_n \in \cP_n$, $n=1,2,\dots$,
  that $\square$-converges to $X$.
  Therefore,
  \[
  \limsup_{n\to\infty} \square(X,\cP_n)
  \le \lim_{n\to\infty} \square(X,X_n) = 0.
  \]

  Let us verify the second condition of Definition \ref{defn:w-conv}.
  Suppose that $\liminf_{n\to\infty} \square(X,\cP_n) = 0$
  for an mm-space $X$.
  It suffices to prove that $X$ belongs to $\cP$.
  We find a subsequence $\{\cP_{n_i}\}$ of $\{\cP_n\}$
  in such a way that $\lim_{i\to\infty} \square(X,\cP_{n_i}) = 0$.
  There is an mm-space $X_i \in \cP_{n_i}$ for each $i$
  such that $X_i$ $\square$-converges to $X$ as $i\to\infty$.
  Therefore, $X$ belongs to $\overline{\cP}_\infty = \cP$.

  This completes the proof.
\end{proof}

\begin{defn}[Metric on the space of pyramids]
  \label{defn:metric-Pi}
  Define,
  for a natural number $k$ and for two pyramids $\cP$ and $\cP'$,
  \begin{align*}
    \rho_k(\cP,\cP') &:= \frac{1}{4k} d_H(\cP\cap\chi(k,k),\cP'\cap\chi(k,k)),\\
    \rho(\cP,\cP') &:= \sum_{k=1}^\infty 2^{-k} \rho_k(\cP,\cP').
  \end{align*}
\end{defn}

Note that $1/(4k)$ in the definition of $\rho_k$ is necessary
for the proof of Theorem \ref{thm:rho-dconc}.

\begin{prop} \label{prop:metric-Pi}
  $\rho$ is a metric on the space $\Pi$ of pyramids
  that is compatible with weak convergence.
  $\Pi$ is compact with respect to $\rho$.
\end{prop}

\begin{proof}
  We first prove that $\rho$ is a metric.
  Since $\square \le 1$, we have $\rho_k \le 1/(4k)$ for each $k$
  and then $\rho \le 1/4$.
  Each $\rho_k$ is a pseudo-metric on $\Pi$ and so is $\rho$.
  If $\rho(\cP,\cP') = 0$ for two pyramids $\cP$ and $\cP'$,
  then $\rho_k(\cP,\cP') = 0$ for any $k$,
  which implies $\cP = \cP'$.
  Thus, $\rho$ is a metric on $\Pi$.

  We next prove the compatibility of the metric $\rho$
  with weak convergence in $\Pi$.
  It follows from Lemma \ref{lem:pyramid-conv-dH} that
  a sequence of pyramids $\cP_n$, $n=1,2,\dots$,
  converges weakly to a pyramid $\cP$
  if and only if $\lim_{n\to\infty} \rho_k(\cP_n,\cP) = 0$
  for any $k$, which is also equivalent to
  $\lim_{n\to\infty} \rho(\cP_n,\cP) = 0$.

  Since $\Pi$ is sequentially compact
  (see Theorem \ref{thm:cpt-pyramid}),
  it is compact with respect to $\rho$.
  This completes the proof.
\end{proof}

The rest of this section is devoted to
the proof of Theorem \ref{thm:rho-dconc}.

\begin{lem} \label{lem:M-dconc}
  Let $X$ and $Y$ be two mm-spaces.
  For any natural number $N$ we have
  \[
  d_H(\cM(X;N),\cM(Y;N)) \le N\cdot\dconc(X,Y),
  \]
  where the Hausdorff distance $d_H$ is defined with respect to
  the Prokhorov metric $d_P$.
\end{lem}

\begin{proof}
  Assume that $\dconc(X,Y) < \varepsilon$ for a real number $\varepsilon$.
  There are two parameters $\varphi : I \to X$ and $\psi : I \to Y$
  such that
  \begin{equation}
    \label{eq:M-dconc}
    d_H(\varphi^*\Lip_1(X),\psi^*\Lip_1(Y)) < \varepsilon.    
  \end{equation}
  Let us prove that $\cM(X;N) \subset B_{N\varepsilon}(\cM(Y;N))$.
  Take any $F_*\mu_X \in \cM(X;N)$, where $F : X \to (\R^N,\|\cdot\|_\infty)$
  is a $1$-Lipschitz map.
  Setting $(f_1,\dots,f_N) := F$ we have $f_i \in \Lip_1(X)$ and so
  $f_i \circ\varphi \in \varphi^*\Lip_1(X)$.
  By \eqref{eq:M-dconc},
  there is a function $g_i \in \Lip_1(Y)$ such that
  $\me(f_i\circ\varphi,g_i\circ\psi) < \varepsilon$.
  Since $G := (g_1,\dots,g_N) : Y \to (\R^N,\|\cdot\|_\infty)$ is $1$-Lipschitz,
  we have $G_*\mu_Y \in \cM(Y;N)$.
  We prove $d_P(F_*\mu_X,G_*\mu_Y) \le N\varepsilon$ in the following.
  For this, it suffices to prove
  $F_*\mu_X(B_\varepsilon(A)) \ge G_*\mu_Y(A) - N\varepsilon$
  for any Borel subset $A \subset \R^N$.
  Since $F_*\mu_X = (F\circ\varphi)_*\cL^1$
  and $G_*\mu_Y = (G\circ\psi)_*\cL^1$, we have
  \[
  F_*\mu_X(B_\varepsilon(A)) = \cL^1((F\circ\varphi)^{-1}(B_\varepsilon(A))),
  \quad
  G_*\mu_Y(A) = \cL^1((G\circ\psi)^{-1}(A)).
  \]
  It is sufficient to prove
  \[
  \cL^1((G\circ\psi)^{-1}(A) \setminus (F\circ\varphi)^{-1}(B_\varepsilon(A)))  \le N\varepsilon.
  \]
  If we take
  $s \in (G\circ\psi)^{-1}(A) \setminus (F\circ\varphi)^{-1}(B_\varepsilon(A))$,
  then
  $G\circ\psi(s) \in A$ and $F \circ\varphi(s) \notin B_\varepsilon(A)$
  together imply
  \[
  \|F\circ\varphi(s) - G\circ\psi(s)\|_\infty > \varepsilon
  \]
  and therefore
  \begin{align*}
    &\cL^1((G\circ\psi)^{-1}(A) \setminus
    (F\circ\varphi)^{-1}(B_\varepsilon(A)))\\
    &\le \cL^1(\{\;s \in I \mid
    \|F\circ\varphi(s) - G\circ\psi(s)\|_\infty > \varepsilon\;\})\\
    &= \cL^1\left( \bigcup_{i=1}^N \{\;s\in I \mid
      |f_i\circ\varphi(s)-g_i\circ\psi(s)| > \varepsilon\;\}\right)\\
    &\le \sum_{i=1}^N \cL^1(\{\;s\in I \mid
    |f_i\circ\varphi(s)-g_i\circ\psi(s)| > \varepsilon\;\})\\
    &\le N\varepsilon,
  \end{align*}
  where the last inequality follows from
  $\me(f_i\circ\varphi,g_i\circ\psi) < \varepsilon$.
  We thus obtain $d_P(F_*\mu_X,G_*\mu_Y) \le N\varepsilon$,
  so that $\cM(X;N) \subset B_{N\varepsilon}(\cM(Y;N))$.
  Since this also holds if we exchange $X$ and $Y$,
  we have
  \[
  d_H(\cM(X;N),\cM(Y;N)) \le N\varepsilon.
  \]
  This completes the proof.
\end{proof}

\begin{lem} \label{lem:MR-half}
  Let $X$ and $Y$ be two mm-spaces.
  Then, for any natural number $N$ and nonnegative real number $R$ we have
  \[
  d_H(\cM(X;N,R),\cM(Y;N,R))
  \le 2\,d_H(\cM(X;N),\cM(Y;N)),
  \]
  where the Hausdorff distance $d_H$ is defined with respect to
  the Prokhorov metric $d_P$ on $\cM(N)$
\end{lem}

\begin{proof}
  Let $\varepsilon := d_H(\cM(X;N),\cM(Y;N))$.
  For any measure $\mu \in \cM(X;N,R)$ there is a measure
  $\nu \in \cM(Y;N)$ such that $d_P(\mu,\nu) \le \varepsilon$.
  This implies
  \begin{equation}
    \label{eq:MR-half}
    \nu(B_\varepsilon(B^N_R)) \ge \mu(B^N_R) - \varepsilon = 1-\varepsilon.
  \end{equation}
  Let $\pi = \pi_R : \R^N \to B^N_R$ be the nearest point projection.
  This is $1$-Lipschitz and satisfies $\pi|_{B^N_R} = \id_{B^N_R}$.
  We have $\pi_*\nu \in \cM(Y;N,R)$.
  By Lemma \ref{lem:dP-me} and \eqref{eq:MR-half}, we see
  $d_P(\pi_*\nu,\nu) \le \me_\nu(\pi,\id_{\R^N}) \le \varepsilon$
  and hence
  \[
  d_P(\mu,\pi_*\nu) \le d_P(\mu,\nu) + d_P(\nu,\pi_*\nu) \le 2\varepsilon,
  \]
  so that $\cM(X;N,R) \subset B_{2\varepsilon}(\cM(Y;N,R))$.
  Exchanging $X$ and $Y$ yields
  $\cM(Y;N,R) \subset B_{2\varepsilon}(\cM(X;N,R))$.
  We thus obtain
  \[
  d_H(\cM(X;N,R),\cM(Y;N,R)) \le 2\varepsilon.
  \]
  This completes the proof.
\end{proof}

\begin{lem} \label{lem:PX-MNR-dH}
  Let $X$ and $Y$ be two mm-spaces, $N$ a natural number,
  and $R$ a nonnegative real number.
  Then we have
  \[
  d_H(\cP_X \cap \cX(N,R),\cP_Y \cap \cX(N,R))
  \le 2\,d_H(\cM(X;N,R),\cM(Y;N,R)).
  \]
\end{lem}

\begin{proof}
  The lemma follows from Proposition \ref{prop:box-di}.
\end{proof}


\begin{proof}[Proof of Theorem \ref{thm:rho-dconc}]
  By Lemmas \ref{lem:PX-MNR-dH}, \ref{lem:M-dconc}, and
  \ref{lem:MR-half},
  \begin{align*}
    & d_H(\cP_X \cap \cX(N,R),\cP_Y \cap \cX(N,R))\\
    &\le 2 d_H(\cM(X;N,R),\cM(Y;N,R))\\
    &\le 4 d_H(\cM(X;N),\cM(Y;N))
    \le 4N \dconc(X,Y),
  \end{align*}
  so that $\rho_k(\cP_X,\cP_Y) \le \dconc(X,Y)$ for any $k$.
  This completes the proof of the theorem.
\end{proof}

\section{Gaussian space and Hopf quotient}
\label{sec:Gaussian-Hopf}

In this section we present the precise definitions of the spaces
appeared in Theorem \ref{thm:main}.

For $\lambda > 0$, let $\gamma^n_{\lambda^2}$ denote
the \emph{$n$-dimensional centered Gaussian measure on $\R^n$
with variance $\lambda^2$}, i.e.,
\[
\gamma^n_{\lambda^2}(A) := \frac{1}{(2\pi\lambda^2)^{\frac{n}{2}}} \int_A e^{-\frac{1}{2\lambda^2}\|x\|_2^2} \; dx
\]
for a Lebesgue measurable subset $A \subset \R^n$,
where $dx$ is the Lebesgue measure on $\R^n$
and $\|\cdot\|_2$ the $l_2$ (or Euclidean) norm on $\R^n$.
We put $\gamma^n := \gamma^n_1$, which is
the \emph{$n$-dimensional standard Gaussian measure on $\R^n$}.
Note that the $n$-th product measure of $\gamma^1_{\lambda^2}$ coincides with $\gamma^n_{\lambda^2}$.
We call the mm-space
$\Gamma^n_{\lambda^2} := (\R^n,\|\cdot\|_2,\gamma^n_{\lambda^2})$
the \emph{$n$-dimensional Gaussian space with variance $\lambda^2$}.
Call $\Gamma^n := \Gamma^n_1$
the \emph{$n$-dimensional standard Gaussian space}.
For $k \le n$, we denote by $\pi^n_k : \R^n \to \R^k$ the natural projection,
i.e.,
\[
\pi^n_k(x_1,x_2,\dots,x_n) := (x_1,x_2,\dots,x_k),
\quad (x_1,x_2,\dots,x_n) \in \R^n.
\]
Since the projection $\pi^n_{n-1} : \Gamma^n_{\lambda^2} \to \Gamma^{n-1}_{\lambda^2}$
is $1$-Lipschitz continuous and measure-preserving for any $n \ge 2$,
the Gaussian space $\Gamma^n_{\lambda^2}$ is monotone
increasing in $n$ with respect to the Lipschitz order,
so that, as $n\to\infty$, the associated pyramid $\cP_{\Gamma^n_{\lambda^2}}$ 
converges to the $\square$-closure of $\bigcup_{n=1}^\infty \cP_{\Gamma^n_{\lambda^2}}$,
denoted by $\cP_{\Gamma^\infty_{\lambda^2}}$.
We call $\cP_{\Gamma^\infty_{\lambda^2}}$
the \emph{virtual infinite-dimensional Gaussian space with variance $\lambda^2$}.
Call $\cP_{\Gamma^\infty} := \cP_{\Gamma^\infty_1}$
the \emph{virtual infinite-dimensional standard Gaussian space}.

Recall that the \emph{Hopf action} is the following $S^1$-action
on $\C^n$:
\[
S^1 \times \C^n \ni (e^{\sqrt{-1}t},z) \longmapsto e^{\sqrt{-1}t} z \in \C^n,
\]
where $S^1$ is the group of unit complex numbers under multiplication.
Since the projection $\pi^{2n}_{2k} : \C^n \to \C^k$, $k \le n$,
is $S^1$-equivariant, i.e.,
$\pi^{2n}_{2k}(e^{\sqrt{-1}t}z) = e^{\sqrt{-1}t} \pi^{2n}_{2k}(z)$
for any $e^{\sqrt{-1}t} \in S^1$ and $z \in \C^n$,
there exists a unique map $\bar{\pi}^{2n}_{2k} : \C^n/S^1 \to \C^k/S^1$
such that the following diagram commutes:
\[
\begin{CD}
  \C^n @>S^1 >> \C^n/S^1\\
  @V \pi^{2n}_{2k} VV @VV \bar{\pi}^{2n}_{2k} V \\
  \C^k @>S^1 >> \C^k/S^1
\end{CD}
\]
We consider the Hopf action on $\Gamma^{2n}_{\lambda^2}$ by
identifying $\R^{2n}$ with $\C^n$.
The Hopf action is isometric with respect to the Euclidean distance
and also preserves the Gaussian measure $\gamma^{2n}_{\lambda^2}$.
Let
\[
\Gamma^{2n}_{\lambda^2}/S^1
= (\C^n/S^1,d_{\C^n/S^1},\bar{\gamma}^{2n}_{\lambda^2})
\]
be the quotient space with the induced mm-structure (see \S\ref{ssec:quotient}).
Note that this is isometric to the Euclidean cone (cf.~\cite{BBI})
over a complex projective space of complex dimension $n-1$
with the Fubini-Study metric.
Since the map $\bar{\pi}^{2n}_{2(n-1)}
: \Gamma^{2n}_{\lambda^2}/S^1 \to \Gamma^{2(n-1)}_{\lambda^2}/S^1$
is $1$-Lipschitz continuous and pushes $\bar{\gamma}^{2n}_{\lambda^2}$
forward to $\bar{\gamma}^{2(n-1)}_{\lambda^2}$,
the quotient space $\Gamma^{2n}_{\lambda^2}/S^1$ is monotone increasing in $n$
with respect to the Lipschitz order.
The associated pyramid $\cP_{\Gamma^{2n}_{\lambda^2}/S^1}$ converges to
the $\square$-closure of $\bigcup_{n=1}^\infty \cP_{\Gamma^{2n}_{\lambda^2}/S^1}$,
which we denote by $\cP_{\Gamma^\infty_{\lambda^2}/S^1}$.
We put $\cP_{\Gamma^\infty/S^1} := \cP_{\Gamma^\infty_1/S^1}$.

\section{Estimate of observable diameter}
\label{sec:normal}

In this section, we give some estimates of the observable diameters
of spheres and complex projective spaces, which are little extensions
of known results (see \cite{Gromov,Ledoux}).

Let $\sigma^n$ denotes the normalized volume measure
on the sphere
\[
S^n(r) := \{\;x \in \R^{n+1} \mid \|x\|_2 = r\;\}
\]
of radius $r > 0$.
For $k \le n$, we consider the restriction of the projection
$\pi^{n+1}_k : S^n(r) \subset \R^{n+1} \to \R^k$,
which is $1$-Lipschitz continuous with respect to
the geodesic distance and also
to the restriction of the Euclidean distance to $S^n(r)$.

Recall that \emph{weak} (resp.~\emph{vague}) \emph{convergence of measures}
means weak-$*$ convergence in the dual of the space of bounded continuous
functions (resp.~continuous functions with compact support).

The following is well-known and the proof is elementary.

\begin{prop}[Maxwell-Boltzmann distribution law\footnotemark]
  \label{prop:MB-law}
  \footnotetext{This is also called the Poincar\'e limit theorem
    in many literature.
    However, there is no evidence that Poincar\'e proved this
    (see \cite{DF}*{\S 6.1}).}
  For any natural number $k$ we have
  \[
  \frac{d(\pi^{n+1}_k)_*\sigma^n}{d\cL^k} \to \frac{d\gamma^k}{d\cL^k}
  \qquad \text{as $n \to \infty$},
  \]
  where $\sigma^n$ is the normalized volume measure on $S^n(\sqrt{n})$
  and $\frac{d}{d\cL^k}$ means the Radon-Nikodym derivative
  with respect to the $k$-dimensional Lebesgue measure.
  In particular,
  \[
  (\pi^{n+1}_k)_*\sigma^n \to \gamma^k \ \text{weakly}
  \qquad \text{as $n \to \infty$}.
  \]
\end{prop}

We prove the following theorem,
since we find no proof in any literature.

\begin{thm}[Normal law \`a la L\'evy; Gromov \cite{Gromov}*{\S 3$\frac12$}]
  \label{thm:normal}
  \ \\ Let $f_n : S^n(\sqrt{n}) \to \R$, $n=1,2,\dots$, be
  $1$-Lipschitz continuous functions with respect to the geodesic distance on
  $S^n(\sqrt{n})$.
  Assume that, for a subsequence $\{f_{n_i}\}$ of $\{f_n\}$,
  the push-forward $(f_{n_i})_*\sigma^{n_i}$ converges vaguely
  to a Borel measure $\sigma_\infty$ on $\R$, and that
  $\sigma_\infty$ is not identically equal to zero.
  Then, $\sigma_\infty$ is a probability measure and
  \[
  (\R,|\cdot|,\sigma_\infty) \prec (\R,|\cdot|,\gamma^1),
  \]
  i.e., there exists a $1$-Lipschitz continuous function $\alpha : \R \to \R$
  such that $\alpha_*\gamma^1 = \sigma_\infty$.
\end{thm}

Note that there always exists a subsequence $\{f_{n_i}\}$
such that $(f_{n_i})_*\sigma^{n_i}$ converges vaguely to
some finite Borel measure on $\R$.

We need some claims for the proof of Theorem \ref{thm:normal}.
The following theorem is well-known.

\begin{thm}[L\'evy's isoperimetric inequality \cite{Levy,FLM}]
  \label{thm:Levy-isop}
  For any closed subset $\Omega \subset S^n(1)$,
  we take a metric ball $B_\Omega$ of $S^n(1)$ with
  $\sigma^n(B_\Omega) = \sigma^n(\Omega)$.
  Then we have
  \[
  \sigma^n(U_r(\Omega)) \ge \sigma^n(U_r(B_\Omega))
  \]
  for any $r > 0$,
  where $U_r(\Omega)$ denotes the open $r$-neighborhood of $\Omega$
  with respect to the geodesic distance on $S^n(\sqrt{n})$.
\end{thm}

We assume the condition of Theorem \ref{thm:normal}.
Consider a natural compactification $\bar\R := \R \cup \{-\infty,+\infty\}$
of $\R$.  Then, by replacing $\{f_{n_i}\}$ with a subsequence,
$\{(f_{n_i})_*\sigma^{n_i}\}$ converges weakly to a probability measure
$\bar{\sigma}_\infty$ on $\bar{\R}$.
We have $\bar{\sigma}_\infty|_{\R} = \sigma_\infty$
and $\sigma_\infty(\R) + \bar{\sigma}_\infty\{-\infty,+\infty\}
= \bar{\sigma}_\infty(\bar{\R}) = 1$.
We prove the following

\begin{lem}\label{lem:normal1}
  Let $x$ and $x'$ be two given real numbers.
  If $\gamma^1(\,-\infty,x\,] = \bar\sigma_\infty[\,-\infty,x'\,]$
  and if $\sigma_\infty\{x'\} = 0$, then
  \[
  \sigma_\infty[\,x'-\varepsilon_1,x'+\varepsilon_2\,]
  \ge \gamma^1[\,x-\varepsilon_1,x+\varepsilon_2\,]
  \]
  for all real numbers $\varepsilon_1,\varepsilon_2 \ge 0$.
  In particular, $\sigma_\infty$ is a probability measure on $\R$
  and $\bar\sigma_\infty\{-\infty,+\infty\} = 0$.
\end{lem}

\begin{proof}
  We set $\Omega_+ := \{\,f_{n_i} \ge x'\,\}$ and
  $\Omega_- := \{\,f_{n_i} \le x'\,\}$.
  We have $\Omega_+ \cup \Omega_- = S^{n_i}(\sqrt{n_i})$.
  Let us prove
  \begin{align}\label{eq:normal-omega1}
    U_{\varepsilon_1}(\Omega_+) \cap U_{\varepsilon_2}(\Omega_-)
    \subset \{\,x'-\varepsilon_1 < f_{n_i} < x'+\varepsilon_2\,\}.
  \end{align}
  In fact, for any point $\xi \in U_{\varepsilon_1}(\Omega_+)$,
  there is a point $\xi' \in \Omega_+$ such that the geodesic distance
  between $\xi$ and $\xi'$ is less than $\varepsilon_1$.
  The $1$-Lipschitz continuity of $f_{n_i}$ proves that
  $f_{n_i}(\xi) > f_{n_i}(\xi') - \varepsilon_1 \ge x'-\varepsilon_1$.
  Thus we have
  $U_{\varepsilon_1}(\Omega_+) \subset \{\,x'-\varepsilon_1 < f_{n_i}\;\}$
  and, in the same way,
  $U_{\varepsilon_2}(\Omega_-) \subset \{\,f_{n_i} < x'+\varepsilon_2\,\}$.
  Combining these two inclusions implies \eqref{eq:normal-omega1}.

  It follows from \eqref{eq:normal-omega1} and
  $U_{\varepsilon_1}(\Omega_+) \cup U_{\varepsilon_2}(\Omega_-)= S^{n_i}(\sqrt{n_i})$
  that
  \begin{align*}
    &(f_{n_i})_*\sigma^{n_i}[\,x'-\varepsilon_1,x'+\varepsilon_2\,]\\
    &= \sigma^{n_i}(x'-\varepsilon_1 \le f_{n_i} \le x'+\varepsilon_2)
    \ge \sigma^{n_i}(U_{\varepsilon_1}(\Omega_+) \cap U_{\varepsilon_2}(\Omega_-))\\
    &= \sigma^{n_i}(U_{\varepsilon_1}(\Omega_+)) + \sigma^{n_i}(U_{\varepsilon_2}(\Omega_-))
    -1,
  \end{align*}
  where $\sigma^{n_i}(P)$ means the $\sigma^{n_i}$-measure of the set of
  points satisfying a conditional formula $P$.
  The L\'evy's isoperimetric inequality (Theorem \ref{thm:Levy-isop})
  implies
  $\sigma^{n_i}(U_{\varepsilon_1}(\Omega_+)) \ge \sigma^{n_i}(U_{\varepsilon_1}(B_{\Omega_+}))$
  and
  $\sigma^{n_i}(U_{\varepsilon_2}(\Omega_-)) \ge \sigma^{n_i}(U_{\varepsilon_2}(B_{\Omega_-}))$,
  so that
  \[
  (f_{n_i})_*\sigma^{n_i}[\,x'-\varepsilon_1,x'+\varepsilon_2\,]
  \ge \sigma^{n_i}(U_{\varepsilon_1}(B_{\Omega_+})) + \sigma^{n_i}(U_{\varepsilon_2}(B_{\Omega_-}))
  -1.
  \]
  It follows from $\sigma_\infty\{x'\} = 0$ that $\sigma^{n_i}(\Omega_+)$ 
  converges to $\bar\sigma_\infty(\,x',+\infty\,]$ as $n\to\infty$.
  We besides have
  $\bar\sigma_\infty(\,x',+\infty\,] = \gamma^1[\,x,+\infty\,) \neq 0,1$.
  Therefore, there is a number $i_0$ such that
  $\sigma^{n_i}(\Omega_+) \neq 0,1$ for all $i \ge i_0$.
  For each $i \ge i_0$ we have a unique real number $a_i$ satisfying
  $\gamma^1[\,a_i,+\infty\,) = \sigma^{n_i}(\Omega_+)$.
  The number $a_i$ converges to $x$ as $i\to\infty$.
  By the Maxwell-Boltzmann distribution law (Proposition \ref{prop:MB-law}),
  \[
  \lim_{i\to\infty} \sigma^{n_i}(U_{\varepsilon_1}(B_{\Omega_+}))
  = \gamma^1[\,x-\varepsilon_1,+\infty\,)
  \]
  as well as
  \[
  \lim_{i\to\infty} \sigma^{n_i}(U_{\varepsilon_2}(B_{\Omega_-}))
  = \gamma^1(\,-\infty,x+\varepsilon_2\,].
  \]
  Therefore,
  \begin{align*}
    \sigma_\infty[\,x'-\varepsilon_1,x'+\varepsilon_2\,]
    &\ge \liminf_{i\to\infty} (f_{n_i})_*\sigma^{n_i}[\,x'-\varepsilon_1,x'+\varepsilon_2\,]\\
    &\ge \gamma^1[\,x-\varepsilon_1,+\infty\,) + \gamma^1(\,-\infty,x+\varepsilon_2\,]
    -1\\
    &= \gamma^1[\,x-\varepsilon_1,x+\varepsilon_2\,].
  \end{align*}
  The first part of the lemma is obtained.

  The rest of the proof is to show that $\sigma_\infty(\R) = 1$.
  Suppose $\sigma_\infty(\R) < 1$.
  Then, since $\sigma_\infty(\R) > 0$,
  there is a non-atomic point $x'$ of $\sigma_\infty$
  such that $0 < \bar\sigma_\infty[\,-\infty,x'\,) < 1$.
  We find a real number $x$ in such a way that
  $\gamma^1(\,-\infty,x\,] = \bar\sigma_\infty[\,-\infty,x'\,]$.
  The first part of the lemma implies
  $\sigma_\infty[\,x'-\varepsilon_1,x'+\varepsilon_2\,]
  \ge \gamma^1[\,x-\varepsilon_1,x+\varepsilon_2\,]$
  for all $\varepsilon_1,\varepsilon_2 \ge 0$.
  Taking the limit as $\varepsilon_1,\varepsilon_2 \to +\infty$,
  we obtain $\sigma_\infty(\R) = 1$.
  This completes the proof.
\end{proof}

\begin{lem} \label{lem:normal2}
  $\supp\sigma_\infty$ is a closed interval.
\end{lem}

\begin{proof}
  $\supp\sigma_\infty$ is a closed set by the definition of the support
  of a measure.
  It then suffices to prove the connectivity of $\supp\sigma_\infty$.
  Suppose not.  Then, there are numbers $x'$ and $\varepsilon > 0$
  such that
  $\sigma_\infty(\,-\infty,x'-\varepsilon\,) > 0$,
  $\sigma_\infty[\,x'-\varepsilon,x'+\varepsilon\,] = 0$,
  and $\sigma_\infty(\,x'+\varepsilon,+\infty\,) > 0$.
  There is a number $x$ such that
  $\gamma^1(\,-\infty,x\,] = \sigma_\infty(\,-\infty,x'\,]$.
  Lemma \ref{lem:normal1} shows that
  $\sigma_\infty[\,x'-\varepsilon,x'+\varepsilon\,]
  \ge \gamma^1[\,x-\varepsilon,x+\varepsilon\,] > 0$,
  which is a contradiction.
  The lemma has been proved.
\end{proof}

\begin{proof}[Proof of Theorem \ref{thm:normal}]
  For any given real number $x$, there exists a smallest number $x'$
  satisfying $\gamma^1(\,-\infty,x\,] \le \sigma_\infty(\,-\infty,x'\,]$.
  The existence of $x'$ follows from the right-continuity and the monotonicity
  of $y \mapsto \sigma_\infty(\,-\infty,y\,]$.
  Setting
  $\alpha(x) := x'$ we have a function
  $\alpha : \R \to \R$, which is monotone nondecreasing.

  We first prove the continuity of $\alpha$ in the following.
  Take any two numbers $x_1$ and $x_2$ with $x_1 < x_2$.
  We have
  $\gamma^1(\,-\infty,x_1\,] \le \sigma_\infty(\,-\infty,\alpha(x_1)\,]$
  and $\gamma^1(\,-\infty,x_2\,] \ge \sigma_\infty(\,-\infty,\alpha(x_2)\,)$,
  which imply
  \begin{align} \label{eq:normal2}
    \gamma^1[\,x_1,x_2\,] \ge \sigma_\infty(\,\alpha(x_1),\alpha(x_2)\,).
  \end{align}
  This shows that, as $x_1 \to a-0$ and $x_2 \to a+0$ for a number $a$,
  we have
  $\sigma_\infty(\,\alpha(x_1),\alpha(x_2)\,) \to 0$, which together with
  Lemma \ref{lem:normal2} implies $\alpha(x_2) - \alpha(x_1) \to 0$.
  Thus, $\alpha$ is continuous on $\R$.

  Let us next prove the $1$-Lipschitz continuity of $\alpha$.
  We take two numbers $x$ and $\varepsilon > 0$ and fix them.
  It suffices to prove that
  \[
  \Delta\alpha := \alpha(x+\varepsilon)-\alpha(x) \le \varepsilon.
  \]

  \begin{clm}
    If $\sigma_\infty\{\alpha(x)\} = 0$, then $\Delta\alpha \le \varepsilon$.
  \end{clm}

  \begin{proof}
    The claim is trivial if $\Delta\alpha = 0$.
    We thus assume $\Delta\alpha > 0$.
    Since $\sigma_\infty\{\alpha(x)\} = 0$, we have
    $\gamma^1(\,-\infty,x\,] = \sigma_\infty(\,-\infty,\alpha(x)\,]$,
    so that Lemma \ref{lem:normal1} implies that
    \begin{align} \label{eq:normal1}
      \sigma_\infty[\,\alpha(x),\alpha(x)+\delta\,]
      \ge \gamma^1[\,x,x+\delta\,]
    \end{align}
    for all $\delta \ge 0$.
    By \eqref{eq:normal2} and \eqref{eq:normal1},
    \begin{align*}
      \gamma^1[\,x,x+\varepsilon\,]
      &\ge \sigma_\infty(\,\alpha(x),\alpha(x+\varepsilon)\,)\\
      &= \sigma_\infty[\,\alpha(x),\alpha(x)+\Delta\alpha\,)\\
      &= \lim_{\delta \to \Delta\alpha-0} \sigma_\infty[\,\alpha(x),\alpha(x)+\delta\,]\\
      &\ge \lim_{\delta \to \Delta\alpha-0} \gamma^1[\,x,x+\delta\,]\\
      &= \gamma^1[\,x,x+\Delta\alpha\,],
    \end{align*}
    which implies $\Delta\alpha \le \varepsilon$.
  \end{proof}

  We next prove that $\Delta\alpha \le \varepsilon$ in the case where
  $\sigma_\infty\{\alpha(x)\} > 0$.
  We may assume that $\Delta\alpha > 0$.
  Let $x_+ := \sup \alpha^{-1}(\alpha(x))$.
  It follows from $\alpha(x) < \alpha(x+\varepsilon)$
  that $x_+ < x+\varepsilon$.
  The continuity of $\alpha$ implies that $\alpha(x_+) = \alpha(x)$.
  There is a sequence of positive numbers
  $\varepsilon_i \to 0$ such that
  $\sigma_\infty\{\alpha(x_++\varepsilon_i)\} = 0$.
  By applying the claim above,
  \[
  \alpha(x_++\varepsilon_i+\varepsilon)-\alpha(x_++\varepsilon_i) \le \varepsilon.
  \]
  Moreover we have
  $\alpha(x+\varepsilon) \le \alpha(x_++\varepsilon_i+\varepsilon)$
  and $\alpha(x_++\varepsilon_i) \to \alpha(x_+) = \alpha(x)$ as $i\to\infty$.
  Thus,
  \[
  \alpha(x+\varepsilon) - \alpha(x) \le \varepsilon
  \]
  and so $\alpha$ is $1$-Lipschitz continuous.

  The rest is to prove that $\alpha_*\gamma^1 = \sigma_\infty$.
  Take any number $x' \in \alpha(\R)$ and fix it.
  Set $x := \sup \alpha^{-1}(x') \;(\le +\infty)$.
  We then have $\alpha(x) = x'$ provided $x < +\infty$.
  Since $x$ is the largest number to satisfy
  $\gamma^1(\,-\infty,x\,] \le \sigma_\infty(\,-\infty,x'\,]$,
  we have $\gamma^1(\,-\infty,x\,] = \sigma_\infty(\,-\infty,x'\,]$,
  where we agree $\gamma^1(\,-\infty,+\infty\,] = 1$.
  By the monotonicity of $\alpha$, we obtain
  \[
  \alpha_*\gamma^1(\,-\infty,x'\,] = \gamma^1(\alpha^{-1}(\,-\infty,x'\,])
  = \gamma^1(\,-\infty,x\,] = \sigma_\infty(\,-\infty,x'\,],
  \]
  which implies that $\alpha_*\gamma^1 = \sigma_\infty$
  because $\sigma_\infty$ is a Borel probability measure.
  This completes the proof.
\end{proof}

\begin{prop} \label{prop:scale-ObsDiam}
  Let $X$ be an mm-space.
  Then, for any real number $t > 0$ and $\kappa \ge 0$, we have
  \[
  \ObsDiam(tX;-\kappa) = t\ObsDiam(X;-\kappa).
  \]
\end{prop}

\begin{proof}
  We have
  \begin{align*}
    &\ObsDiam(tX;-\kappa)\\
    &= \sup\{\;\diam(f_*\mu_X;1-\kappa) \mid
    \text{$f : tX \to \R$ $1$-Lipschitz}\;\}\\
    &= \sup\{\;\diam(f_*\mu_X;1-\kappa) \mid
    \text{$t^{-1}f : X \to \R$ $1$-Lipschitz}\;\}\\
    &= \sup\{\;\diam((tg)_*\mu_X;1-\kappa) \mid
    \text{$g : X \to \R$ $1$-Lipschitz}\;\}\\
    &= t\ObsDiam(X;-\kappa).
  \end{align*}
  This completes the proof.
\end{proof}

\begin{cor} \label{cor:normal}
  Let $\{r_n\}_{n=1}^\infty$ be a sequence of positive real numbers.
  If $r_n/\sqrt{n} \to \lambda$ as $n\to\infty$ for a real number $\lambda$, then
  \[
  \lim_{n\to\infty} \ObsDiam(S^n(r_n);-\kappa)
  = \diam(\gamma^1_{\lambda^2};1-\kappa) = 2\lambda I^{-1}((1-\kappa)/2),
  \]
  for any real number $\kappa$ with $0 < \kappa < 1$,
  where $I(r) := \gamma^1[\,0,r\,]$.
  This holds for the geodesic distance function on $S^n(r_n)$
  and also for the restriction to $S^n(r_n)$ of
  the Euclidean distance function on $\R^n$.
\end{cor}

\begin{proof}
  It suffices to prove the corollary in the case where $r_n = \sqrt{n}$,
  because of Proposition \ref{prop:scale-ObsDiam}.
  We assume that $r_n = \sqrt{n}$.

  The corollary for the geodesic distance function
  follows from the normal law \`a la L\'evy
  (Theorem \ref{thm:normal})
  and the Maxwell-Boltzmann distribution law (Proposition \ref{prop:MB-law}).

  Assume that the distance function on $S^n(\sqrt{n})$ to be
  Euclidean.
  Since the Euclidean distance is not greater than the geodesic distance,
  the normal law \`a la L\'evy still holds.
  Clearly, the projection $\pi^{n+1}_k : S^n(\sqrt{n}) \to \R^k$ is $1$-Lipschitz
  with respect to the Euclidean distance function.
  Thus, the corollary follows from the normal law \`a la L\'evy
  and the Maxwell-Boltzmann distribution law.
  This completes the proof.
\end{proof}

\begin{lem} \label{lem:dP-quotient}
  Let $X$ be a metric space and $G$ a group acting on $X$ isometrically.
  Then, for any two Borel probability measures $\mu$ and $\nu$ on $X$,
  we have
  \[
  d_P(\bar{\mu},\bar{\nu}) \le d_P(\mu,\nu),
  \]
  where $\bar{\mu}$ denotes the push-forward of $\mu$
  by the projection $X \to X/G$.
\end{lem}

\begin{proof}
  Take any Borel subset $\bar{A} \subset X/G$ and set $A := \pi^{-1}(\bar{A})$,
  where $\pi : X \to X/G$ is the projection.
  Assume that $d_P(\mu,\nu) < \varepsilon$ for a real number $\varepsilon$.
  Since $\pi^{-1}(B_\varepsilon(\bar{A})) \supset B_\varepsilon(A)$, we have
  \[
  \bar{\mu}(B_\varepsilon(\bar{A})) = \mu(\pi^{-1}(B_\varepsilon(\bar{A})))
  \ge \mu(B_\varepsilon(A)) \ge \nu(A)-\varepsilon
  = \bar{\nu}(\bar{A})-\varepsilon
  \]
  and therefore $d_P(\bar{\mu},\bar{\nu}) \le \varepsilon$.
  This completes the proof.
\end{proof}

Denote by $\zeta^n$ the normalized volume measure on $\CP^n(r)$
with respect to the Fubini-Study metric.
Note that $\zeta_n = \bar{\sigma}^{2n+1}$, where
$\bar{\sigma}^{2n+1}$ denotes the push-forward of $\sigma^{2n+1}$
by the Hopf fibration $S^{2n+1}(r) \to \CP^n(r)$.

\begin{prop} \label{prop:MB-law-CPn}
  Let $\{r_n\}_{i=1}^\infty$ be a sequence of positive real numbers
  such that $r_n/\sqrt{2n+1} \to \lambda$ as $n\to\infty$
  for a real number $\lambda$.
  Then, for any natural number $k$ we have
  \[
  (\bar{\pi}^{2n+2}_{2k})_*\zeta^n \to \bar{\gamma}^{2k}_{\lambda^2}
  \ \text{weakly}  \quad \text{as $n \to \infty$},
  \]
  where $\bar{\pi}^{2n+2}_{2k} : \CP^n(r_n) \subset \C^{n+1}/S^1 \to \C^k/S^1$
  is as in \S\ref{sec:Gaussian-Hopf} and $\bar{\gamma}^{2k}_{\lambda^2}$
  the push-forward of $\gamma^{2k}_{\lambda^2}$ by the projection
  $\C^k \to \C^k/S^1$.
\end{prop}

\begin{proof}
  By Proposition \ref{prop:MB-law},
  $(\pi^{2n+2}_{2k})_*\sigma^{2n}$ converges weakly to
  $\gamma^{2k}_{\lambda^2}$ as $n\to\infty$.
  In the following, a measure with upper bar means the push-forward
  by the projection to the Hopf quotient.
  Since $(\bar{\pi}^{2n+2}_{2k})_*\zeta^n = (\bar{\pi}^{2n+2}_{2k})_*\bar{\sigma}^{2n}
  = \overline{(\pi^{2n+2}_{2k})_*\sigma^{2n}}$,
  and by Lemma \ref{lem:dP-quotient}, we have
  \[
  d_P((\bar{\pi}^{2n+2}_{2k})_*\zeta^n,\bar{\gamma}^{2k}_{\lambda^2})
  \le d_P((\pi^{2n+2}_{2k})_*\sigma^{2n},\gamma^{2k}_{\lambda^2})
  \to 0 \quad\text{as $n\to\infty$.}
  \]
  This completes the proof.
\end{proof}

\begin{cor} \label{cor:ObsDiam-CPn}
  Let $\{r_n\}_{i=1}^\infty$ be a sequence of positive real numbers.
  If $r_n/\sqrt{2n+1} \to \lambda$ as $n\to\infty$ for a real number $\lambda$, then
  \begin{align*}
    \limsup_{n\to\infty} \ObsDiam(\CP^n(r_n);-\kappa)
    &\le 2 \lambda I^{-1}((1-\kappa)/2),\\
    \liminf_{n\to\infty} \ObsDiam(\CP^n(r_n);-\kappa)
    &\ge \lambda \diam(([\,0,+\infty\,),re^{-\frac{r^2}{2}}dr);1-\kappa).
  \end{align*}
  These inequalities both hold for the scaled Fubini-Study metric
  and also for the induced distance function from
  the Euclidean distance on $S^{2n+1}(r_n)$.
\end{cor}

\begin{proof}
  The upper estimate
  follows from $\CP^n(r_n) \prec S^{2n}(r_n)$
  and Corollary \ref{cor:normal}.

  Since $\bar{\pi}^{2n+2}_2 : \CP^n(r_n) \subset \C^{n+1}/S^1 \to \C/S^1$
  is $1$-Lipschitz, Proposition \ref{prop:MB-law-CPn} proves
  \begin{align*}
    &\liminf_{n\to\infty} \ObsDiam(\CP^n(r_n);-\kappa)
    \ge \diam((\C/S^1,\bar{\gamma}^2_{\lambda^2});1-\kappa)\\
    &= \lambda \diam(([\,0,+\infty\,),re^{-\frac{r^2}{2}}dr);1-\kappa).
  \end{align*}
  This completes the proof.
\end{proof}

We conjecture that
\[
\lim_{n\to\infty} \ObsDiam(\CP^n(r_n);-\kappa)
= \lambda \diam(([\,0,+\infty\,),re^{-\frac{r^2}{2}}dr);1-\kappa)
\]
if $r_n/\sqrt{2n+1} \to \lambda$ as $n\to\infty$.

\section{Limits of spheres
and complex projective spaces}
\label{sec:Spheres-Gaussians}

In this section, we prove Theorem \ref{thm:main}
by using several claims proved before.
We need some more lemmas.

\begin{lem} \label{lem:sphere-Gaussian}
  For any real number $\theta$ with $0 < \theta < 1$, we have
  \[
  \lim_{n\to\infty} \gamma^{n+1}\{\;x\in\R^{n+1}
  \mid \theta\sqrt{n} \le \|x\|_2 \le \theta^{-1}\sqrt{n}\;\} = 1.
  \]
\end{lem}

\begin{proof}
  Considering the poler coordinates on $\R^n$,
  we see that
  \[
  \gamma^{n+1}\{\;x\in\R^{n+1} \mid \|x\|_2 \le r\;\}
  = \frac{\int_0^r t^ne^{-t^2/2}\,dt}{\int_0^\infty t^ne^{-t^2/2}\,dt}.
  \]
  Integrating the both sides of $(\log(t^n e^{-t^2/2}))'' = -n/t^2-1 \le -1$
  over $[\,t,\sqrt{n}\,]$ with $0 < t \le \sqrt{n}$ yields
  \[
  -(\log(t^n e^{-t^2/2}))'
  = (\log(t^n e^{-t^2/2}))'|_{t=\sqrt{n}} - (\log(t^n e^{-t^2/2}))'
  \le t-\sqrt{n}.
  \]
  Integrating this again over $[\,t,\sqrt{n}\,]$ implies
  \[
  \log(t^ne^{-t^2/2}) - \log(n^{n/2}e^{-n/2}) \le -\frac{(t-\sqrt{n})^2}{2},
  \]
  so that $t^ne^{-t^2/2} \le n^{n/2}e^{-n/2}e^{-(t-\sqrt{n})^2/2}$ and then,
  for any $r$ with $0 \le r \le \sqrt{n}$,
  \[
  \int_0^{\sqrt{n}-r} t^ne^{-t^2/2}\,dt
  \le n^{n/2}e^{-n/2} \int_r^{\sqrt{n}} e^{-t^2/2}\,dt
  \le n^{n/2}e^{-n/2}e^{-r^2/2}.
  \]
  Stirling's approximation implies
  \[
  \int_0^\infty t^ne^{-t^2/2}\,dt = 2^{\frac{n-1}{2}} \int_0^\infty s^{\frac{n-1}{2}} e^{-s}\;ds
  	\approx \sqrt{\pi} (n-1)^{\frac{n}{2}} e^{-\frac{n-1}{2}}.
  \]
  Therefore,
  \[
  \lim_{n\to\infty} \gamma^{n+1}\{\;x\in\R^{n+1} \mid \|x\|_2 \le \theta\sqrt{n}\;\}
  \le \lim_{n\to\infty} \frac{n^{n/2}e^{-n/2}e^{-(1-\theta)^2n/2}}
  {\sqrt{\pi} (n-1)^{\frac{n}{2}} e^{-\frac{n-1}{2}}} = 0.
  \]
  The same calculation leads us to obtain
  \[
  \lim_{n\to\infty} \gamma^{n+1}\{\;x\in\R^{n+1} \mid \|x\|_2 \ge \theta^{-1}\sqrt{n}\;\} = 0.
  \]
  This completes the proof.
\end{proof}

For a positive real number $t$ and a pyramid $\cP$,
we put
\[
t\cP := \{\;tX \mid X \in \cP\;\},
\]
where $tX := (X,t\,d_X,\mu_X)$.
The following lemma is obvious and the proof is omitted.

\begin{lem} \label{lem:mul-pyramid}
  Assume that a sequence of pyramids $\cP_n$, $n=1,2,\dots$,
  converges weakly to a pyramid $\cP$,
  and that a sequence of positive real numbers $t_n$, $n=1,2,\dots$,
  converges to a positive real number $t$.
  Then, $t_n\cP_n$ converges weakly to $t\cP$ as $n\to\infty$.
\end{lem}

\begin{proof}[Proof of Theorem \ref{thm:main}]
  (1) follows from Corollaries \ref{cor:normal} and \ref{cor:ObsDiam-CPn}.

  We prove (2) for $S^n(r_n)$.
  We first assume that $r_n/\sqrt{n} \to +\infty$ as $n\to\infty$.
  Take any finitely many positive real numbers
  $\kappa_0,\kappa_1,\dots,\kappa_N$ with $\sum_{i=0}^N \kappa_i < 1$,
  and fix them.
  We find positive real numbers
  $\kappa_0',\kappa_1',\dots,\kappa_N'$
  in such a way that $\kappa_i < \kappa_i'$ for any $i$
  and $\sum_{i=0}^N \kappa_i' < 1$.
  For any $\varepsilon > 0$,
  there are Borel subsets $A_0,A_1,\dots,A_N \subset \R$ such that
  $\gamma^1(A_i) \ge \kappa_i'$ for any $i$
  and
  \[
  \min_{i\neq j} d_{\R}(A_i,A_j)
  > \Sep((\R,\gamma^1);\kappa_0',\dots,\kappa_N')-\varepsilon.
  \]
  Without loss of generality we may assume that all $A_i$'s are open.
  Let $f_n := \pi^n_1|_{S^n(\sqrt{n})} : S^n(\sqrt{n}) \to \R$.
  Proposition \ref{prop:MB-law} tells us that
  $(f_n)_*\sigma^n$ converges weakly to $\gamma^1$ as $n\to\infty$.
  Since $A_i$ is open,
  \[
  \liminf_{n\to\infty} \sigma^n((f_n)^{-1}(A_i)) \ge \gamma^1(A_i)
  \ge \kappa_i' > \kappa_i.
  \]
  There is a natural number $n_0$ such that
  $\sigma^n((f_n)^{-1}(A_i)) \ge \kappa_i$ for any $i$ and
  $n \ge n_0$.
  Since $f_n$ is $1$-Lipschitz continuous for the Riemannian metric
  and also for the Euclidean distance on $S^n(\sqrt{n})$,
  we have
  \[
  d_{S^n(\sqrt{n})}((f_n)^{-1}(A_i),(f_n)^{-1}(A_j)) \ge d_{\R}(A_i,A_j).
  \]
  Therefore, for any $n \ge n_0$,
  \[
  \Sep(S^n(\sqrt{n});\kappa_0,\dots,\kappa_N)
  > \Sep((\R,\gamma^1);\kappa_0',\dots,\kappa_N')-\varepsilon,
  \]
  which proves
  \[
  \liminf_{n\to\infty} \Sep(S^n(\sqrt{n});\kappa_0,\dots,\kappa_N)
  \ge \Sep((\R,\gamma^1);\kappa_0',\dots,\kappa_N') > 0.
  \]
  Since $r_n/\sqrt{n} \to \infty$ as $n\to\infty$,
  \begin{equation}
    \label{eq:Levy-dissipate-sphere}
    \Sep(S^n(r_n);\kappa_0,\dots,\kappa_N)
    = \frac{r_n}{\sqrt{n}} \Sep(S^n(\sqrt{n});\kappa_0,\dots,\kappa_N)
  \end{equation}
  is divergent to infinity and so $\{S^n(r_n)\}_{n=1}^\infty$
  infinitely dissipates.

  We next prove the converse.
  Assume that $\{S^n(r_n)\}$ infinitely dissipates
  and $r_n/\sqrt{n}$ is not divergent to infinity.
  Then, there is a subsequence $\{r_{n(j)}\}$ of $\{r_n\}$
  such that $r_{n(j)}/\sqrt{n(j)}$ is bounded for all $j$.
  By \eqref{eq:Levy-dissipate-sphere},
  $\{S^{n(j)}(\sqrt{n(j)})\}_j$ infinitely dissipates.
  However, for each fixed $\kappa$ with $0 < \kappa < 1/2$,
  $\ObsDiam(S^n(\sqrt{n});-\kappa)$ is bounded for all $n$
  by Corollary \ref{cor:normal},
  and, by Proposition \ref{prop:ObsDiam-Sep},
  so is $\Sep(S^n(\sqrt{n});\kappa,\kappa)$,
  which contradicts
  that $\{S^{n(j)}(\sqrt{n(j)})\}$ infinitely dissipates.
  This completes the proof of (2).

  (2) for $\CP^n(r_n)$ is proved in the same way as for $S^n(r_n)$
  by using Proposition \ref{prop:MB-law-CPn} and
  Corollary \ref{cor:ObsDiam-CPn} instead of
  Proposition \ref{prop:MB-law} and Corollary \ref{cor:normal}.

  We prove (3) for $S^n(r_n)$.
  By Lemma \ref{lem:mul-pyramid} and
  by $\Gamma^n_{\lambda^2} = \lambda\Gamma^n$,
  it suffices to prove it in the case of $r_n = \sqrt{n}$.
  We assume that $r_n = \sqrt{n}$.
  Suppose that $\cP_{S^n(\sqrt{n})}$ does not converge weakly to
  $\cP_{\Gamma^\infty}$ as $n\to\infty$.
  Then, by the compactness of $\Pi$, there is a subsequence
  $\{\cP_{S^{n_i}(\sqrt{n_i})}\}$ of $\{\cP_{S^n(\sqrt{n})}\}$
  that converges weakly to a pyramid $\cP$ with $\cP \neq \cP_{\Gamma^\infty}$.
  It follows from the Maxwell-Boltzmann distribution law
  (Proposition \ref{prop:MB-law}) that
  $\Gamma^k$ belongs to $\cP$ for any $k$,
  so that $\cP_{\Gamma^\infty} \subset \cP$.
  We take any real number $\theta$ with $0 < \theta < 1$ and fix it.
  It follows from Lemma \ref{lem:mul-pyramid}
  that $\cP_{S^{n_i}(\theta\sqrt{n_i})}$ converges weakly to $\theta\cP$
  as $i\to\infty$.
  Define a function $f_{\theta,n} : \R^{n+1} \to \R^{n+1}$ by
  \[
  f_{\theta,n}(x) :=
  \begin{cases}
    \frac{\theta\sqrt{n}}{\|x\|_2} x & \text{if $\|x\|_2 > \theta\sqrt{n}$},\\
    x & \text{if $\|x\|_2 \le \theta\sqrt{n}$},
  \end{cases}
  \]
  for $x \in \R^{n+1}$.
  $f_{\theta,n}$ is $1$-Lipschitz continuous with respect to
  the Euclidean distance.
  Let $\sigma_\theta^n$ be the normalized volume measure on
  $S^n(\theta\sqrt{n})$.  We consider $\sigma_\theta^n$
  as a measure on $\R^{n+1}$ via the natural embedding
  $S^n(\theta\sqrt{n}) \subset \R^{n+1}$.
  From Lemma \ref{lem:sphere-Gaussian}, we have
  \[
  d_P((f_{\theta,n})_*\gamma^{n+1},\sigma_\theta^n)
  \le \gamma^{n+1}\{\;x\in\R^{n+1} \mid \|x\|_2 < \theta\sqrt{n}\;\}
  \to 0 \ \text{as $n\to\infty$},
  \]
  so that the box distance between
  $S_{\theta,n} := (\R^{n+1},\|\cdot\|_2,(f_{\theta,n})_*\gamma^{n+1})$
  and $S^n(\theta\sqrt{n})$
  converges to zero as $n\to\infty$.
  By Proposition \ref{prop:dconc-box} and Theorem \ref{thm:rho-dconc},
  we have $\rho(\cP_{S_{\theta,n}},\cP_{S^n(\theta\sqrt{n})}) \to 0$
  as $n\to\infty$.
  Therefore, $\cP_{S_{\theta,n_i}}$ converges weakly to $\theta\cP$
  as $i\to\infty$.  Since $S_{\theta,n} \prec (\R^{n+1},\|\cdot\|_2,\gamma^{n+1})$,
  we have
  $\cP_{S_{\theta,n}} \subset \cP_{\Gamma^{n+1}} \subset \cP_{\Gamma^\infty}$.
  We thus obtain $\theta\cP \subset \cP_{\Gamma^\infty} \subset \cP$
  for any $\theta$ with $0 < \theta < 1$ and so
  $\cP = \cP_{\Gamma^\infty}$, which is a contradiction.
  This completes the proof of (3) for $S^n(r_n)$.

  We prove (3) for $\CP^n(r_n)$.
  The proof is similar to that for $S^n(r_n)$.
  We may assume that $r_n = \sqrt{2n+1}$.
  Suppose that $\cP_{\CP^n(\sqrt{2n+1})}$ does not converge weakly to
  $\cP_{\Gamma^\infty/S^1}$ as $n\to\infty$.
  Then, by the compactness of $\Pi$, there is a subsequence
  $\{\CP^{n(j)}(\sqrt{2n(j)+1})\}_j$ of $\{\CP^n(\sqrt{2n+1})\}_n$
  that converges weakly to a pyramid $\cP$ with
  $\cP \neq \cP_{\Gamma^\infty/S^1}$.
  Proposition \ref{prop:MB-law-CPn} proves
  that $\Gamma^{2k}/S^1 \subset \cP$ for any $k$ and
  so $\cP_{\Gamma^\infty/S^1} \subset \cP$.
  We take any real number $\theta$ with $0 < \theta < 1$ and fix it.
  By Lemma \ref{lem:mul-pyramid},
  $\cP_{\CP^{n(j)}(\theta \sqrt{2n(j)+1})}$
  converges weakly to $\theta\cP$ as $j\to\infty$.
  Let $f_{\theta,2n+1} : \C^{n+1} \to \C^{n+1}$ be as above
  by identifying $\C^{n+1}$ with $\R^{2n+2}$.
  $f_{\theta,2n+1}$ is $S^1$-equivariant
  and induces a map $\bar{f}_{\theta,n} : \C^{n+1}/S^1 \to \C^{n+1}/S^1$,
  which is $1$-Lipschitz continuous.
  Let $\sigma^{2n+1}_\theta$ be the normalized volume measure on
  $S^{2n+1}(\theta\sqrt{2n+1})$.
  Since $(f_{\theta,2n+1})_*\gamma^{2n+2}|_{S^{2n+1}(\theta\sqrt{2n+1})}$
  is a constant multiple of $\sigma^{2n+1}_\theta$,
  the measure $\overline{(f_{\theta,2n+1})_*\gamma^{2n+2}}|_{\CP^n(\theta\sqrt{2n+1})}$
  is also a constant multiple of $\bar{\sigma}^{2n+1}_\theta$,
  where the upper bar means the push-forward of a measure
  by the projection to the Hopf quotient space.
  We see that
  \[
  \overline{(f_{\theta,2n+1})_*\gamma^{2n+2}}
  = (\bar{f}_{\theta,2n+2})_*\bar{\gamma}^{2n+2}.
  \]
  By Lemmas \ref{lem:dP-quotient} and \ref{lem:sphere-Gaussian}, we have
  \[
  d_P((\bar{f}_{\theta,2n+2})_*\bar{\gamma}^{2n+2},\bar{\sigma}^{2n+1}_\theta)
  \le d_P((f_{\theta,2n+1})_*\gamma^{2n+2},\sigma^{2n+1}_\theta)
  \to 0 \quad\text{as $n\to\infty$},
  \]
  so that the box distance between
  $Y_{\theta,n} :=
  (\C^{n+1},\|\cdot\|_2,(\bar{f}_{\theta,n})_*\bar{\gamma}^{2n+2})$
  and $\CP^n \{\theta\sqrt{2n+1}\}$ converges to zero as $n\to\infty$.
  The rest of the proof is same as before.
  This completes the proof of the theorem.
\end{proof}

\section{$\dconc$-Cauchy property and box convergence}

In this section, we prove Proposition \ref{prop:dconc}
and prove the non-convergence property for spheres and complex projective spaces
with respect to the box distance.

\begin{prop} \label{prop:no-concentrated}
  The virtual infinite-dimensional standard Gaussian space $\cP_{\Gamma^\infty}$
  is not concentrated.
\end{prop}

\begin{proof}
  Let $\varphi_{n,i} : \R^n \to \R$, $i=1,2,\dots,n$, be the functions defined by
  \[
  \varphi_{n,i}(x_1,x_2,\dots,x_n) := x_i,
  \qquad (x_1,x_2,\dots,x_n) \in \R^n.
  \]
  Each $\varphi_{n,i}$ is $1$-Lipschitz continuous
  and we have the elements $[\varphi_{n,i}]$ of $\Lo(\Gamma^n)$.
  For any different $i$ and $j$,
  \begin{align*}
    \me_{\gamma^n}([\varphi_{n,i}],[\varphi_{n,j}])
    = \me_{\gamma^2}([\varphi_{2,1}],[\varphi_{2,2}])
    = \me_{\gamma^2}(\varphi_{2,1},\varphi_{2,2}+t)
  \end{align*}
  for some real number $t$.
  If $\me_{\gamma^2}(\varphi_{2,1},\varphi_{2,2}+t) = 0$ were to hold,
  then $\varphi_{2,1} = \varphi_{2,2} + t$ almost everywhere,
  which is a contradiction.
  Thus, $\me_{\gamma^n}([\varphi_{n,i}],[\varphi_{n,j}])$ is a positive
  constant independent of $n$, $i$, and $j$ with $i \neq j$.
  This implies that $\{\Lo(\Gamma^n)\}_{n=1}^\infty$
  is not $d_{GH}$-precompact.
  Since $\Gamma^n \in \cP_{\Gamma^\infty}$,
  the pyramid $\cP_{\Gamma^\infty}$ is not concentrated.
  This completes the proof.
\end{proof}

As a direct consequence of Proposition \ref{prop:no-concentrated},
Lemma \ref{lem:concentrated}, and Theorem \ref{thm:main},
we have the following.

\begin{cor} \label{cor:no-concentrated}
  \begin{enumerate}
  \item $\{\Gamma^n\}_{n=1}^\infty$ has no $\dconc$-Cauchy subsequence.
  \item If $\{r_n\}$ is a sequence of positive real numbers
    with $r_n/\sqrt{n} \to 1$ as $n\to\infty$,
    then $\{S^n(r_n)\}$ has no $\dconc$-Cauchy subsequence.
  \end{enumerate}
\end{cor}

\begin{proof}
  (1) follows from Proposition \ref{prop:no-concentrated} and
  Lemma \ref{lem:concentrated}.

  Theorem \ref{thm:main}(3) implies that $\cP_{S^n(r_n)}$ converges weakly
  to $\cP_{\Gamma^\infty}$ as $n\to\infty$, which together with
  Proposition \ref{prop:no-concentrated} and
  Lemma \ref{lem:concentrated} proves (2) of the corollary.
  This completes the proof.
\end{proof}

\begin{proof}[Proof of Proposition \ref{prop:dconc}]
  Since $\{S^n(\sqrt{n})\}$ has no $\dconc$-Cauchy subsequence,
  there exist two subsequences $\{X_n\}$ and $\{Y_n\}$ of $\{S^n(\sqrt{n})\}$
  such that $\inf_n \dconc(X_n,Y_n) > 0$.
  Theorem \ref{thm:main}(3) implies that
  $\cP_{X_n}$ and $\cP_{Y_n}$ both converge weakly to $\cP_{\Gamma^\infty}$.
  This completes the proof.
\end{proof}

\begin{lem} \label{lem:dom-precpt}
  Let $X_n$ and $Y_n$, $n=1,2,\dots$, be mm-spaces such that $X_n \prec Y_n$
  for any $n$.
  If $\{Y_n\}$ is $\square$-precompact, then so is $\{X_n\}$.
  In particular, if $\{Y_n\}$ is $\square$-precompact
  and if $X_n$ concentrates to an mm-space $X$,
  then $X_n$ $\square$-converges to $X$.
\end{lem}

\begin{proof}
  Recall (see \cite{Gromov}*{\S 3$\frac12$.14} and \cite{Shioya:book}*{\S 4})
  that $\{Y_n\}$ is $\square$-precompact if and only if
  for any $\varepsilon > 0$ there exists a number $\Delta(\varepsilon) > 0$
  such that we have
  Borel subsets $K_{n1},K_{n2},\dots,K_{nN} \subset Y_n$ for each $n$ 
  with the property that
  \begin{enumerate}
  \item[(i)] $N \le \Delta(\varepsilon)$;
  \item[(ii)] $\diam K_{ni} \le \varepsilon$ for any $i=1,2,\dots,N$;
  \item[(iii)] $\diam \bigcup_{i=1}^n K_{ni} \le \Delta(\varepsilon)$;
  \item[(iv)] $\mu_{Y_n}(\bigcup_{i=1}^n K_{ni}) \ge 1-\varepsilon$.
  \end{enumerate}
  We assume that $\{Y_n\}$ is $\square$-precompact, and then
  have Borel subsets $K_{ni} \subset Y_n$ satisfying (i)--(iv).
  Without loss of generality we may assume that all $K_{ni}$ are compact,
  since each $\mu_{Y_n}$ is inner regular.
  By $X_n \prec Y_n$, we find a $1$-Lipschitz continuous map
  $f_n : Y_n \to X_n$ with $(f_n)_*\mu_{Y_n} = \mu_{X_n}$.
  The sets $K_{ni}' := f_n(K_{ni})$ are compact and satisfy
  (i)--(iv), so that $\{X_n\}$ is $\square$-precompact.
  The first part of the lemma has been proved.

  We prove the second part.
  Assume that $\{Y_n\}$ is $\square$-precompact
  and that $X_n$ concentrates to an mm-space $X$.
  If $X_n$ does not concentrate to $X$,
  then the $\square$-precompactness of $\{X_n\}$ proves
  that it has a $\square$-convergent subsequence
  whose limit is different from $X$.
  This contradicts that $X_n$ concentrates to $X$ as $n\to\infty$
  (see Proposition \ref{prop:dconc-box}).
  The proof of the lemma is completed.
\end{proof}

\begin{prop} \label{prop:no-box}
  Let $\{r_n\}_{i=1}^\infty$ be a sequence of positive real numbers.
  If $r_n$ is bounded away from zero,
  then $\{S^n(r_n)\}_{n=1}^\infty$ and $\{\CP^n(r_n)\}_{n=1}^\infty$
  both have no $\square$-convergent subsequence.
\end{prop}

\begin{proof}
  Assume that $r_n \ge c > 0$ for any natural number $n$ and
  for a constant $c$.
  We have $S^n(c) \prec S^n(r_n)$ and
  $\CP^n(c) \prec \CP^n(r_n)$.
  According to \cite{Funano:est-box},
  the two sequences $\{S^n(c)\}$ and $\{\CP^n(c)\}$ both have
  no $\square$-convergent subsequence.
  By Lemma \ref{lem:dom-precpt},
  $\{S^n(r_n)\}$ and $\{\CP^n(r_n)\}$ also have no $\square$-convergent
  subsequence.
  This completes the proof.
\end{proof}

\end{document}